\documentclass[12pt]{amsart}
\usepackage{graphicx, overpic}
\usepackage[below]{placeins}
\usepackage[colorlinks=true, linkcolor=blue, citecolor=blue]{hyperref}

\usepackage[T1]{fontenc}

\usepackage{amsmath,amscd,amssymb, epsfig,psfrag,subfigure}

\usepackage{enumerate, amsfonts, latexsym, color, url}
\usepackage{epstopdf}

\input xy
\xyoption{all}
\newcommand{\cd}[1]{\begin{equation*}{\xymatrix{#1}}\end{equation*}}
\newcommand{\cdlabel}[2]{\begin{equation}\label{#1}{\xymatrix{#2}}\end{equation}}

\usepackage{amsthm}

\makeatletter
\newtheorem*{rep@theorem}{\rep@title}
\newcommand{\newreptheorem}[2]{%
\newenvironment{rep#1}[1]{%
 \def\rep@title{#2 \ref{##1}}%
 \begin{rep@theorem}}%
 {\end{rep@theorem}}}
\makeatother

\newreptheorem{theorem}{Theorem}
\newreptheorem{lemma}{Lemma}

\newtheorem{theorem}{Theorem}[section]
\newtheorem{lemma}[theorem]{Lemma}

\newtheorem{cor}[theorem]{Corollary}
\theoremstyle{definition}
\newtheorem{definition}[theorem]{Definition}

\newtheorem{example}[theorem]{Example}

\newtheorem {proposition} [theorem] {Proposition}

\def\onto{\twoheadrightarrow}
\def\into{\hookrightarrow}

\def\min{\mbox{\text{min}}}

\def\min{\mbox{\rm{min}}}

\newtheorem {corollary} [theorem] {Corollary}

\newtheorem {case}{Case}[theorem]

\newtheorem {remark} [theorem] {Remark}

\newtheorem {notation} [theorem] {Notation}

\def\ker {\mathrm{ker}}

\def\mc {\mathcal}

\def\Stab {\mathrm{Stab}}

\newcommand{\bR}{\mathbb{R}}
\newcommand{\bN}{\mathbb{N}}
\newcommand{\Fix}{\mathrm{Fix}}

\newcommand{\co}{\colon\thinspace}

\def\split{\backslash\backslash}

\edef\t@mp{\catcode`\noexpand\#=\the\catcode`\#}%
    \catcode`\#=12
    \def\h@sh{#}%
    \t@mp

\edef\t@mp{\catcode`\noexpand\~=\the\catcode`\~}%
    \catcode`\~=12
    \def\tild@{~}%
    \t@mp

\begin{document}

\title{The virtual Haken conjecture} 

\author[Ian Agol]{%
        Ian Agol\\ with an appendix by Ian Agol, Daniel Groves, and Jason Manning} 
\address{%
          University of California, Berkeley\\
970 Evans Hall \#3840\\
Berkeley, CA, 94720-3840}
    \email{%
        ianagol@math.berkeley.edu}  
\dedicatory{Dedicated to Mike Freedman on the occasion of his 60th birthday}
\thanks{Agol  supported by DMS-0806027,  DMS-1105738 and the Clay Foundation}

\address{%
Department of Mathematics, Statistics, and Computer Science\\
University of Illinois at Chicago\\
322 Science and Engineering Offices (M/C 249)\\
851 S. Morgan St.\\
Chicago, IL 60607-7045
}
\email{groves@math.uic.edu}
\thanks{ Groves supported by DMS-0804365
and CAREER DMS-0953794}
\address{%
244 Mathematics Building\\
Dept. of Mathematics\\
University at Buffalo\\
Buffalo, NY 14260-2900
}
\email{j399m@buffalo.edu}
\thanks{Manning supported by DMS-1104703}

\date{%
\today}


\begin{abstract} 
We prove that cubulated hyperbolic groups are virtually special. 
The proof relies on results of Haglund and Wise which also imply that they are linear  groups, and
quasi-convex subgroups are separable.
A consequence is that closed hyperbolic 3-manifolds have finite-sheeted Haken covers, which 
resolves the virtual Haken question of Waldhausen and Thurston's virtual fibering question. 
An appendix to this paper by Agol, Groves, and Manning proves a generalization of the main
result of \cite{AGM08}. 
\end{abstract} 

\maketitle
\section{Introduction}
In this paper, we will be interested in fundamental groups of
non-positively curved (NPC) cube complexes.

\begin{theorem}\cite[Problem 11.7]{HaglundWise07} \cite[Conjecture 19.5]{Wise11} \label{virtuallyspecial}
Let $G$ be a word-hyperbolic group acting properly and cocompactly on a CAT(0)
cube complex $X$. Then $G$ has a finite index subgroup $F$ acting specially on $X$.
\end{theorem}
{\bf Remark:} The condition here that $G$ is word-hyperbolic is necessary, since there
are examples of simple groups acting properly cocompactly on a product of trees \cite{BurgerMozes}.

\begin{cor} \label{lerf}
Let $G$ be a non-elementary word-hyperbolic group acting properly and cocompactly on a CAT(0)
cube complex. The $G$ is linear, large,
and quasi-convex subgroups are separable. 
\end{cor}

\begin{reptheorem}{virtual Haken}
Let $M$ be a closed aspherical 3-manifold. Then there is a finite-sheeted
cover $\tilde{M}\to M$ such that $\tilde{M}$ is Haken.
\end{reptheorem}
\begin{reptheorem}{virtual fiber}
Let $M$ be a closed hyperbolic 3-manifold. Then there is a finite-sheeted
cover $\tilde{M}\to M$ such that $\tilde{M}$ fibers over the circle. Moreover,
$\pi_1(M)$ is LERF and large.
\end{reptheorem}

Theorem \ref{virtual Haken} resolves a question of Waldhausen \cite{Waldhausen}. Moreover, Theorem \ref{virtual fiber}
resolves \cite[Problems 3.50-51]{Kirby} from Kirby's problem list, as well as \cite[Questions 15-18]{Th:82}. 

There has been much work on the virtual Haken conjecture before for certain classes of manifolds. These include manifolds in the Snappea census \cite{DT03}, surgeries
on various classes of cusped hyperbolic manifolds \cite{Baker89, Baker90, Baker91,  BZ00,  CL, CooperWalsh06I, CooperWalsh06II,KojimaLong88, Masters00, Masters07},
certain arithmetic hyperbolic 3-manifolds (see \cite{Schwermer04} and references therein), and manifolds satisfying various group-theoretic criteria \cite{Lackenby06, Lackenby07, Long87}. 
The approach in this paper uses techniques from geometric group theory, and as such does not specifically rely on 3-manifold techniques,
 although some of the arguments (such as word-hyperbolic Dehn surgery and hierarchies) are inspired by 3-manifold techniques.

Here is a short summary of the approach to the proof of Theorem \ref{virtuallyspecial}. In Section \ref{quotient} we use a weak separability result (Theorem \ref{filling}) to
find an infinite-sheeted regular cover $\mathcal{X}$ of $X/G$ which has embedded compact 2-sided walls. This covering space has a finite
hierarchy obtained by labeling the walls with finitely many numbers (which we think of as colors), so that walls with the same color do not intersect, and cutting successively
along the walls ordered by their labels to get an infinite collection of ``cubical polyhedra''. The goal is to construct a finite-sheeted cover which is ``modeled'' on this hierarchy for $\mathcal{X}$.
We first construct a measure on the space of colorings of the wall graph of $\mathcal{X}$ in Section \ref{colorings}. We then refine the colors to reflect how
each wall is cut up by previous stages of the hierarchy in Section \ref{boundary pattern}. We use the measure to find a solution
to certain gluing equations on the colored cubical polyhedra defined by the refined colorings, and use solutions to these equations to get the base case of the
hierarchy in Section \ref{gluing equations}. We glue up successively each stage of the hierarchy, using a gluing theorem \ref{virtualgluing} to glue at each stage
after passing to a finite-sheeted cover. The inductive  hypotheses and inductive step of the proof of Theorem \ref{virtuallyspecial} are given in
section \ref{virtual gluing}. 

Theorem \ref{filling} generalizes the main result of \cite{AGM08}, and is proved  in the appendix which is
joint work with Groves and Manning. The proof of Theorem \ref{filling} relies on Theorem \ref{malnormal} which is a result of Wise \cite[Theorem 12.3]{Wise11}.

{\bf Acknowledgements:} We thank Nathan Dunfield and Dani Wise for helpful conversations, and
Martin Bridson, Fr\'ederic Haglund, Yi Liu, Eduardo Martinez-Pedroza, and Henry Wilton for comments on an earlier draft. 

\section{Definitions}
We expect the reader to be familiar with non-positively curved (NPC) cube 
complexes \cite{BridsonHaefliger}, special cube complexes \cite{HaglundWise07}, and hyperbolic groups \cite{gromov:wordhyperbolic}.

\begin{definition}
A flag simplicial complex is a complex determined by its 1-skeleton: for every
clique (complete subgraph) of the 1-skeleton, there is a simplex with 1-skeleton
equal to that subgraph. 
A non-positively curved (NPC) cube complex is a cube complex $X$ such that
for each vertex $v\in X$, the link $link_X(v)$ is a flag simplicial complex. If $X$
is simply-connected, then $X$ is CAT(0) \cite{BridsonHaefliger}. 
More generally, an NPC cube {\it orbicomplex} or  {\it orbihedron} is a pair $(G,X)$, where each 
component of $X$ is
a CAT(0) cube complex and $G\to Aut(X)$ is a proper cocompact effective action.
We will also call such pairs {\it cubulated groups} when $X$ is connected. 
If $G$ is torsion-free, then $X/G$ is an NPC cube complex. When $G$
has torsion, we may also think of the quotient $X/G$ as an orbi-space
in the sense of Haefliger \cite{Haefliger90, Haefliger91}. The   orbihedra
we will consider in this paper will have  covering spaces which are
cube complexes, so they are {\it developable}, in which case we can ignore subtleties arising in the theory of 
general orbi-spaces. 
\end{definition}
Gluing the cubes isometrically out of unit Euclidean cubes gives a canonical
metric on an NPC cube complex. 
\begin{definition}
Given an NPC cube complex $X$, the {\it wall} of $X$ is an immersed NPC
cube complex $W$ (possibly disconnected). For each $n$-cube $C\subset X$, take the $n$
$n-1$-cubes obtained by cutting the cube in half (setting one coordinate $=0$), called the {\it hyperplanes}
of $C$. If a $k$-cube $D$
is a face of an $n$-cube $C$, then there is a corresponding embedding
of the hyperplanes of $D$ as faces of the hyperplanes of $C$. Take
the cube complex $W$ with cubes given by hyperplanes of the cubes of $X$,
and gluings given by inclusion of hyperplanes. This cube complex immerses into the
cube complex $X$. We will call this immersed cube complex the {\it wall complex} of $X$.
There is a natural line bundle over $W$ obtained by piecing together the
normal bundles in each cube. If this line bundle is non-orientable, then 
the wall $W$ is one-sided. Otherwise, it is 2-sided or co-orientable, and
there are two possible co-orientations.  
\end{definition}
\begin{definition} 
Let $X$ be an NPC cube complex. 
A subcomplex $Y\subset X$ is {\it locally convex} if the embedding $Y\to X$
is a local isometry. Similarly, a combinatorial map $Y\looparrowright X$ between
NPC cube complexes is called locally convex if it is a local isometry. 
\end{definition}
The condition of being a local isometry is equivalent to saying that $Y$ is NPC, and
for each vertex $v\in Y$,
$link_Y(v)\subset link_X(v)$ is a very full subcomplex, which means that for any two 
vertices of $link_Y(v)$ which are joined by an edge in $link_X(v)$, they are also joined by an edge in $link_Y(v)$.
For example, an embedded cube in an NPC cube complex is a locally convex subcomplex.
\begin{definition}  [Almost malnormal Collection] 
A collection of subgroups $H_1,\ldots, H_g$ of $G$ is almost malnormal
provided that $|H_i^g \cap H_j| < \infty$ unless $i = j$ and $g \in H_i$.
\end{definition}
For example, finite collections maximal elementary subgroups of a torsion-free hyperbolic group form an almost
malnormal collection.
\begin{definition}
Let $X$ be an NPC cube complex, $Y \subset X$ a locally convex subcomplex. We say
that $Y$ is {\it acylindrical} if for any map $(S^1\times[0,1],S^1\times\{0,1\})\to (X,Y)$ which is
injective on $\pi_1$  is relatively homotopic to a map $(S^1\times[0,1],S^1\times\{0,1\})\to (Y,Y)$.
\end{definition}

In particular, if $Y\subset X$ is acylindrical, then the collection of subgroups of the
fundamental groups of its components form a malnormal collection of subgroups of $\pi_1(X)$.

\begin{definition}(\cite[Definition 11.5]{Wise11}) \label{qvh}
Let $\mathcal{QVH}$ denote the smallest class of hyperbolic groups that is closed under the following
operations.
\begin{enumerate}
\item $1\in \mathcal{QVH}$
\item  If $G = A  \ast_B C$ and $A, C \in \mathcal{QVH}$ and $B$ is finitely generated and embeds by a quasi-isometry in $G$, then $G$ is in $\mathcal{QVH}$.
\item If $G = A \ast_B$ and $A \in \mathcal{QVH}$ and $B$ is f.g. and embeds by a quasi-isometry, then $G$ is in $\mathcal{QVH}$.
\item Let $H < G$ with $[G : H] < \infty$. If $H \in \mathcal{QVH}$  then $G \in \mathcal{QVH}$.
\end{enumerate}
\end{definition}
The notation $\mathcal{QVH}$ is meant to be an abbreviation for ``quasi-convex virtual hierarchy''. 
In particular, items (2) and (3) may be replaced by $G$ is a graph of groups, with vertex groups in $\mathcal{QVH}$
and edge groups quasi-convex and f.g. in $G$. A motivating example of a group in $\mathcal{QVH}$ is
a 1-relator group with torsion \cite[Theorem 18.1]{Wise11}, or a closed-hyperbolic 3-manifold
which contains an embedded quasi-fuchsian surface. 

We will not define special cube complexes in this paper (see \cite{HaglundWise07}). However, we note that a special cube complex with hyperbolic fundamental group
has embedded components of the wall subcomplex, and therefore its fundamental group is in the class $\mathcal{QVH}$.
Moreover, we have
\begin{theorem}\cite[Theorem 13.5]{Wise11}
A torsion-free hyperbolic group is in $\mathcal{QVH}$ if and only if it is the fundamental group of a virtually
special cube complex. 
\end{theorem}
This theorem is generalized in the appendix:
\begin{reptheorem}{QVH}
A word-hyperbolic group is in  $\mc{QVH}$ if and only if it is virtually special. 
\end{reptheorem}
The reader not familiar with virtually  special cube complexes or their fundamental groups may therefore
take this theorem as the defining property for a virtually special group which will be used in this paper.

We state here a lemma which will be used in the case that $G$ has torsion.

\begin{lemma} \label{finite extension}
Let $G\in \mathcal{QVH}$, and suppose that $G'$ is an extension of $G$ by a finite group $K<G'$,
so there is a homomorphism $\varphi: G'\to G$ such that $ker(\varphi)=K$. Then $G'\in \mathcal{QVH}$. 
\end{lemma}
\begin{proof}
First, we remark that we need only check this for central extensions of $G$, since the kernel of 
the  homomorphism $G'\to Aut(K)$ given by the conjugacy action will be a finite central extension of a finite-index
subgroup of $G$. However, this observation does not seem to simplify the argument.

We prove this by induction on the defining properties of $\mathcal{QVH}$. Consider the set of
all extensions of groups in $\mathcal{QVH}$ by the  finite group $K$. We prove that this class lies
in $\mathcal{QVH}$ by showing that these groups are closed
under the four operations defining $\mathcal{QVH}$. 

\begin{enumerate}
\item The extension of $1$ by $K$ is in $\mathcal{QVH}$ by property (4), so $K\in \mathcal{QVH}$. 

\item  Suppose $G = A  \ast_B C$ and $A, C \in \mathcal{QVH}$ and $B$ is finitely generated and embeds by a quasi-isometry in $G$.

We see that for $A'=\varphi^{-1}(A), B'=\varphi^{-1}(B), C'=\varphi^{-1}(C)$, then $\varphi$ is a quasi-isometry, so $B'<G'$ embeds
quasi-isometrically in $G'$, and is finitely generated since $B$ is. Also, $A', B', C'$ are finite extensions of $A,B, C$ by $K$ respectively. 
If $A', B', C' \in \mathcal{QVH}$, then $G'\in \mathcal{QVH}$ by condition (2). 

\item Suppose $G = A \ast_B$ and $A \in \mathcal{QVH}$ and $B$ is f.g. and embeds by a quasi-isometry.

We see that for $A'=\varphi^{-1}(A), B'=\varphi^{-1}(B)$, then $\varphi$ is a quasi-isometry, so $B'<G'$ embeds
quasi-isometrically in $G'$, and is finitely generated since $B$ is. Also, $A', B'$ are finite extensions of $A,B$ by $K$ respectively. 
If $A', B' \in \mathcal{QVH}$, then $G'\in \mathcal{QVH}$ by condition (3).

\item Suppose $H < G$ with $[G : H] < \infty$ and $H \in \mathcal{QVH}$.

Then for $H'=\varphi^{-1}(H)$, we have $[G':H']=[G:H]< \infty$, so if $H'\in \mathcal{QVH}$, and $H'$ is a
finite extension of $H$ by $K$, then $G'\in \mathcal{QVH}$ by condition (4).
\end{enumerate}
Thus, we see that finite extensions of elements of $\mathcal{QVH}$ by $K$ are in $\mathcal{QVH}$ by induction,
so $G'\in \mathcal{QVH}$. 
\end{proof}

\section{Virtual gluing}

In this section, we introduce a technical theorem which will be used in the proof of Theorem \ref{virtuallyspecial}. 
\begin{theorem} \label{virtualgluing}
Let $X$ be a compact cube complex which is virtually special and $\pi_1(X)$ hyperbolic (for each
component of $X$ if $X$ is disconnected). 
Let $Y \subset X$ be an embedded locally convex  acylindrical subcomplex such
that there is an NPC cube orbi-complex $Y_0$ and a cover $\pi: Y\to Y_0$ . 
Then there exists a regular cover $\overline{X}\to X$ such that the preimage of $Y\leftarrow\overline{Y}\subset \overline{X}$
is a  regular orbi-cover  $\overline{Y}\to Y_0$. 
\end{theorem}
{\bf Remark:} Keep in mind that all of the complexes in the statement of this theorem may be disconnected.
\begin{proof}
Let $C\pi$ be the mapping cylinder of $\pi$. 
Take $X'=X\cup_{Y} C\pi$ (we'll assume now that $X'$ is connected; the general case reduces to this case by considering 
components). Then $X'$ is an NPC cube orbi-complex by \cite[Theorem 11.1]{BridsonHaefliger}, and the subspace $Y_0\subset X'$ is 
locally convex. By \cite{Haefliger90}, the complex $X'$ is developable, so $X'= \tilde{X}'/G'$ for a hyperbolic group $G'\cong \pi_1(X')$. 
Moreover, since $Y$ is acylindrical, $Y_0\subset X'$ is also acylindrical (the reader uncomfortable with orbispaces may 
just think of $G'=\pi_1(X')$ as an acylindrical graph of groups with quasiconvex edge groups, and apply \cite{Kapovich01}). 
The subspace
$Y\subset X'$ is $\pi_1$-injective, and cutting along it gives back $X$ and $C\pi$. The cone $C\pi\simeq Y_0$
is virtually special, since $Y\subset X$ is virtually special being a convex subcomplex of a virtually special
complex with hyperbolic fundamental group, and $Y\to Y_0$ is a covering orbi-space. We may think of $C\pi \to Y_0$ as a $\ast$-bundle over $Y_0$,
where $\ast$ is a wedge of $deg(\pi)$ intervals. 
By the combination theorem of Bestvina-Feighn \cite[Corollary 7]{BestvinaFeighn92}, $\pi_1(X')$ is hyperbolic. 
Therefore $X'$ is in $\mathcal{QVH}$  and 
virtually special by Theorem \ref{QVH}. Let $Y'\to Y_0$ be a regular covering space
factoring through $Y\to Y_0$. Then there exists a finite-sheeted cover  $X''\to X'$ in which
$Y'$ lifts to an embedding since each component of $Y'$ is quasi-convex in $X'$, and therefore is separable by \cite[Corollary 7.4]{HaglundWise07}.
In particular, the preimage of $C\pi$ meeting $Y'$ in $X''$ is a product $Y' \times \ast$, where $\ast$ is a wedge of intervals. 
Taking a further regular cover of $X'$ gives a covering space in which 
the induced cover of $Y$ is a regular cover of $Y_0$. 
\end{proof}

{\bf Remark:} It is possible to give a proof of this theorem using the techniques of \cite[Theorem 6.1]{HaglundWise09} rather than citing \cite{Wise11}.

\section{Quotient complex with compact walls} \label{quotient}
Let $X$ be a CAT(0) cube complex, $G$ a hyperbolic group acting properly and cocompactly
on $X$.  Recall that we say that $(G,X)$ is a cubulated hyperbolic group. 
Since the action of $G$ is proper and cocompact, the cube complex $X$ is finite dimensional,
locally finite, and quasi-isometric to $G$. By Lemma \ref{finite extension}, we may assume that
$G$ acts faithfully on $X$, since properness implies that the subgroup acting trivially must
be finite. 
The quotient $X/G$ may be interpreted as an orbihedron \cite{Haefliger90, Haefliger91} if $G$ has torsion. 
Moreover, there are finitely many orbits of walls $W\subset X$. 
The stabilizer of a wall $G_W$ is quasi-isometric to $W$, and therefore is a quasi-convex subgroup
of $G$ since $W$ is totally geodesic and therefore convex in $X$. Let $\{W_1,\ldots, W_m\}$ be
orbit representatives for the walls of $X$ under the action of $G$. By induction on the maximal
dimension of a cube and Lemma \ref{finite extension}, we may
assume that $G_{W_i}$ is virtually special for $1\leq i \leq m$. In particular, for each $i$,
there is a finite index torsion-free normal subgroup $G_i' \dot{\trianglelefteq} G_{W_i}$ such
that $W_i/G_i'$ is a special cube complex. There exists $R>0$ such that if two walls
$W, W' \subset X$ have the property that $d(W,W')> R$, then $|G_W\cap G_{W'}|<\infty$. 

For each $1\leq i \leq m$, let 
$$\mathcal{A}_i = \{ G_{W_i} g G_{W_i} | d(g(W_i),W_i)\leq R\}-\{G_{W_i}\}.$$
Then $\mathcal{A}_i$ is finite for all $i$. 

\begin{lemma}
We may find a quotient
group homomorphism $\phi: G\to \mathcal{G}$ such that for all $1\leq i \leq m$ and  for all $G_{W_i}g G_{W_i}\in \mathcal{A}_i$,
$\phi(g)\notin \phi(G_{W_i})$
and $\phi(G_{W_j})$ is finite for all $j$. Moreover, we may assume that the action of $G_{W_i}\cap ker(\phi)$ does not
exchange the sides of $W_i$ (preserves the co-orientation), and that $ker(\phi)$ is torsion-free and $X^{(1)}/ker(\phi)$ contains no closed loops. 
\end{lemma}
\begin{proof}
For each $W_i$, the set of double cosets $\mathcal{A}_i=\{G_{W_i}g G_{W_i}, d(g(W_i),W_i)\leq 2R\}-\{G_{W_i}\}$ is finite.
Fix an element $g$ such that $G_{W_i}g G_{W_i} \in \mathcal{A}_i$. Choose elements $g_1,\ldots, g_m$
such that $g_i=1$ and $H = \langle G_{W_1}^{g_1}, \ldots, G_{W_i}, \ldots, G_{W_m}^{g_m}  \rangle \cong G_{W_1}^{g_1}\ast \cdots \ast G_{W_i} \ast \cdots \ast G_{W_w}^{g_w}$ and $g\notin H$, and $H<G$ is quasiconvex. This may be arranged by standard ping-pong arguments. Then $H$ is virtually
special since it is a free product of virtually special groups. 
By Theorem \ref{filling},  we may find a quotient $\phi_g: G \to \mathcal{G}_g$ such that $\phi_g(g)\notin \phi(H)$
and $\phi_g(H)$ is finite. Clearly then $\phi_g(G_{W_j})$ is finite for all $j$. Moreover, we may assume
that $ker(\phi_g)\cap G_{W_i}$ is contained in the subgroup preserving the orientation on the normal
bundle to $W_i$. Let $\mathcal{A}$ be the finitely many double coset representatives for $\cup_i \mathcal{A}_i$
we use in this construction. 

Let $\mathcal{T}\subset G$ be a finite set of representatives for each conjugacy class of torsion elements 
of $G$, such that $\mathcal{T}\cap G_{W_j}=\emptyset$ for all $j$, and conjugacy class representatives
of any group elements identifying endpoints of edges of $X^{(1)}$. 
We may also apply the same technique to find for each $g\in \mathcal{T} $   a homomorphism  
$\psi_g: G\to \mathcal{G}'$ such that $\psi_g(g)\neq 1$ and $\psi_g(G_{W_i})$ is
finite for all $i$. 

Define $\phi$ by $ker(\phi)=\cap_{g \in \mathcal{A}} ker(\phi_g) \cap_{g \in \mathcal{T}} ker(\psi_g)$,
then $\phi: G\to \mathcal{G}=G/ker(\phi)$ has the desired properties. 
\end{proof}

Let $K=ker(\phi)$, where $\phi$ comes from the previous lemma, and let $\mathcal{X}=X/K$. Then $\mathcal{X}$ is
an NPC cube complex, and for each $i$, the quotient $\mathcal{N}_R(W_i)/ (G_{W_i}\cap K)$ embeds
in $\mathcal{X}$ under the natural covering map, where $\mathcal{N}_R(W_i)$ is the neighborhood of 
radius $R$ about $W_i$.  

\begin{definition} \label{crossing graph}
Form a graph $\Gamma(\mathcal{X})$, with vertices $V(\Gamma(\mathcal{X}))$   consisting of the 
wall components of $\mathcal{W}\subset \mathcal{X}$,
and edges $E(\Gamma(\mathcal{X}))$ consisting of pairs of walls $(W_1,W_2)$ in $\mathcal{X}$ such
that $d(W_1,W_2)\leq R$. We have a natural action of $\mathcal{G}$ on $\Gamma(\mathcal{X})$. 
\end{definition}

\section{Invariant coloring measures} \label{colorings}
Let $\Gamma$ be a (simplicial) graph of bounded valence $\leq k$, and let $G$
be a group acting  cocompactly on $\Gamma$. Note that the quotient graph $\Gamma/G$ may have loops and multi-edges, so in particular may not be simplicial. We will denote the vertices
of $\Gamma$ by $V(\Gamma)$, and the edges by $E(\Gamma) \subset V(\Gamma)\times V(\Gamma)$ consisting of the symmetric relation of  pairs of adjacent vertices in $\Gamma$, so that $\Gamma$ is defined by the pair $\Gamma=(V(\Gamma),E(\Gamma))$. Since $\Gamma$ is simplicial, it has no loops,
and therefore $E(\Gamma)$ does not meet the diagonal of $V(\Gamma)\times V(\Gamma)$.
\begin{definition}
An  $n$-coloring of $\Gamma$ is a map $c: V(\Gamma)\to \{1,\ldots,n\}=[n]$ 
such that for every edge $(u,v)\in E(\Gamma)$, we have $c(u)\neq c(v)$. 
Let $[n]^{V(\Gamma)}$ be the space of all $n$-colorings of the trivial graph $(V(\Gamma),\emptyset)$, and endow this with the product topology to make it a compact space (Cantor set). 
Then the space of $n$-colorings of $\Gamma$ is naturally a closed $G$-invariant subspace of $[n]^{V(\Gamma)}$ which we will denote $C_n(\Gamma)$. The set $M(C_n(\Gamma))$ of probability
measures on $C_n(\Gamma)$ endowed with the weak*  topology is a convex compact
metrizable set. Let $M_G(C_n(\Gamma))\subset M(C_n(\Gamma))\subset M([n]^{V(\Gamma)})$ denote the $G$-invariant measures. 
\end{definition} 

Since we have assumed that the degree of every vertex of $\Gamma$ is $\leq k$,
then clearly $C_{k+1}(\Gamma)$ is non-trivial: order the vertices, and color each vertex inductively by one of $k+1$ colors not already used by one of its $\leq k$ neighbors. 

\begin{theorem} \label{color measure}
The set $M_G(C_{k+1}(\Gamma))$ is non-empty, that is, there exists a $G$-invariant probability
measure on the space of $k+1$-colorings of the graph $\Gamma$. 
\end{theorem}

\begin{proof}
For a $G$-invariant measure $\nu\in M_G([n]^{V(\Gamma)})$, we want to define a quantity
which measures how far $\nu$ is from giving a $G$-invariant coloring measure in $M_G(C_n(\Gamma))$. 
For an edge $e=(u,v)\in E(\Gamma)$, let $B_e = \{ f\in [n]^{V(\Gamma)} | f(u)=f(v)\}$.
This is the subset of colorings of $V(\Gamma)$ which violate the coloring condition
for $\Gamma$ at the edge $e$, so that $C_n(\Gamma)= \cap_{e\in E(\Gamma)} B_e^{c}$. 
Let $\{e_1,\ldots,e_m\}\subset E(\Gamma)$ be a complete set of representatives of the orbits
of the action of $G$ on $E(\Gamma)$, which exists because we have assumed
that the action of $G$ on $\Gamma$ is co-compact. 
For $\nu\in M_G([n]^{V(\Gamma)})$ define $weight(\nu)= \sum_{i=1}^{m} \nu(B_{e_i})$. 
If $\nu$ is a $G$-invariant coloring measure of $\Gamma$, then regarding  $\nu \in M_G(C_n(\Gamma))\subset M_G([n]^{V(\Gamma)})$, we have 
$weight(\nu)=0$. Conversely, if $weight(\nu)=0$ for $\nu\in M_G([n]^{V(\Gamma)})$, then 
$\nu\in M_G(C_n(\Gamma))$. To see this, let $supp(\nu)\subset [n]^{V(\Gamma)}$ be the support
of $\nu$, which is $\cap_{C\ \text{compact}, \nu(C)=1} C$. Let $e\in E(\Gamma)$, then $\nu(B_{e})=0$, since there
exists $e_i, g\in G$ such that $e=g(e_i)$, so $\nu(B_e)=\nu(B_{g(e_i)}) = \nu(B_{e_i})=0$ by $G$-invariance of $\nu$ and $weight(\nu)=0$.
Therefore $supp(\nu)\subset B_{e}^c$ for all $e\in E(\Gamma)$, and therefore $supp(\nu)\subset \cap_{e\in E(\Gamma)} B_e^{c}=C_n(\Gamma)$.
So $\nu\in M_G(C_n(\Gamma))$. 

Take the uniform measure $\mu_{n}$  on $[n]^{V(\Gamma)}$, which
is the product of $V(\Gamma)$ copies of the uniform measure on $[n]$, so $\mu_n\in M([n]^{V(\Gamma)})$. 
Clearly $\mu_n$ is $G$-invariant under the action of $G$ on $V(\Gamma)$, $\mu_n\in M_G([n]^{V(\Gamma)})$, since
$G$ permutes the uniform measures on $[n]$.
We note that for the uniform measure $\mu_n$, we have
$\mu_n(B_e)=1/n$. 
Then we see that $weight(\mu_n)= m/n$. 

For  $n>k+1$, we define a map $p_n: [n]^{V(\Gamma)} \to [n-1]^{V(\Gamma)}$ which depends on $\Gamma$ and which
is $G$-equivariant. For $c\in [n]^{V(\Gamma)}$ and $v\in V(\Gamma)$, define $p_n(c)(v)= c(v)$ if $c(v)<n$, and if $c(v)= n$, then $p_n(c)(v)=  \min( \{1,\ldots,n-1\} - \{c(u) | (u,v) \in E(\Gamma)\})$. Since the degree of $v$ is $\leq k$, this set is non-empty, 
and  has a well-defined minimum which is $\leq k+1$. In other words, $p_n(c)$ assigns to each vertex colored $n$ the smallest color not used
by its neighbors, and otherwise does not change the color. 
Then $p_n(c)$ has the property that for any two vertices
$u,v\in V(\Gamma)$ with $p_n(c)(u)=p_n(c)(v)$, then $c(u)=c(v)$. In
particular, if $c$ is an $n$-coloring of $\Gamma$, then  $p_n(c)$ is an $n-1$-coloring
of $\Gamma$. This implies that for all measures $\nu \in M_G([n]^{V(\Gamma)})$,
 $weight(p_{n*}(\nu))\leq weight(\nu)$, where $p_{n*}(\nu)$ is the push-forward measure. 
 Notice that the map $p_n$ is continuous, since its definition is local, so that the push-forward
 is well-defined. 

This gives a map $P_n: [n]^{V(\Gamma)}\to [k+1]^{V(\Gamma)}$ defined
by $P_n= p_{k+1}\circ p_{k+2}\circ \cdots \circ p_n$. 
We get induced a map $P_{n*}:M([n]^{V(\Gamma)})\to M([k+1]^{V(\Gamma)})$ by push-forward of measures, and induces a map by restriction
$P_{n*}:M_G([n]^{V(\Gamma)})\to M_G([k+1]^{V(\Gamma)})$ because the maps $p_n$ are $G$-equivariant. 

Finally, we have $weight(P_{n*}(\mu))\leq weight(\mu)$ for any $\mu \in M_G([n]^{V(\Gamma)})$. 
In particular, $weight(P_{n*}(\mu_n)) \leq weight(\mu_n)=m/n$. Take a subsequence
of $\{P_{n*}(\mu_n)\}$ converging to a $G$-invariant measure $\mu_{\infty}\in M_G([k+1]^{V(\Gamma)})$. Then $weight(\mu_{\infty})=0$, which implies that $\mu_{\infty}\in M_G(C_{k+1}(\Gamma))$. 
\end{proof} 

\section{Cube complexes with boundary patterns} \label{boundary pattern}

Given a locally finite cube complex $X$, subdivide each $n$-cube into $2^n$ cubes
of half the size to get a cube complex $\dot{X}$ (Figure \ref{cube}(b)). This is called the {\it cubical barycentric subdivision}, 
and is analogous to the barycentric subdivision
of a complex, in that one inserts new vertices in the barycenter of each cube, and connects each new barycenter
vertex of each cube to the barycenter vertex of each cube containing it, then filling in cubes using
the flag condition (the difference with
the usual barycentric subdivision is that one does not connect the vertices of $X$ to the new
barycenter vertices, which would give a simplicial complex). 
We may then regard the union of the hyperplanes  $W\subset X$ as the union of the
new topological codimension-one cubes of $\dot{X}$, which is the locally convex subcomplex $\dot{W}\subset \dot{X}$ spanned by
the barycenter vertices of $\dot{X}$. Consider splitting $X$ along the hyperplanes $W$. 
By this, we mean remove each hyperplane, getting a disconnected complex, then put in 
$2^k$ copies of each codimension-$k$ cube that is removed to get a complex $X\split W=\dot{X}\split \dot{W}$ (see Figure \ref{cube}(c)). We will think of this as a cube
complex ``with boundary'', where the boundary consists of the new cubes that were attached at the missing hyperplanes. What remains are stars of the vertices of $X$. 

\begin{figure}[htb] 
	\begin{center}
	\subfigure[Cube in $X$]{\epsfig{figure=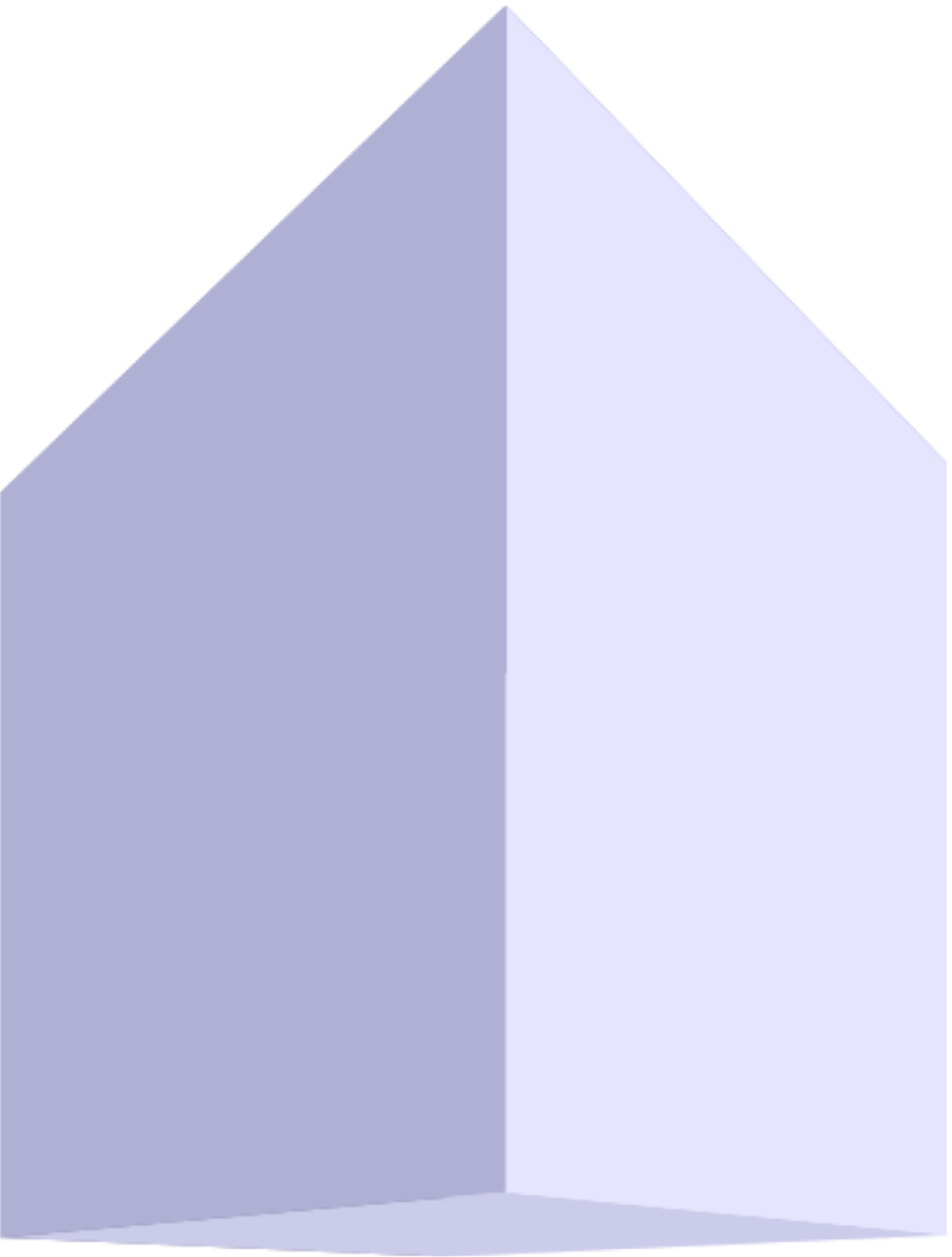,angle=0,width=.3\textwidth}}\quad
	\subfigure[Cubical barycentric subdivision in $\dot{X}$]{\epsfig{figure=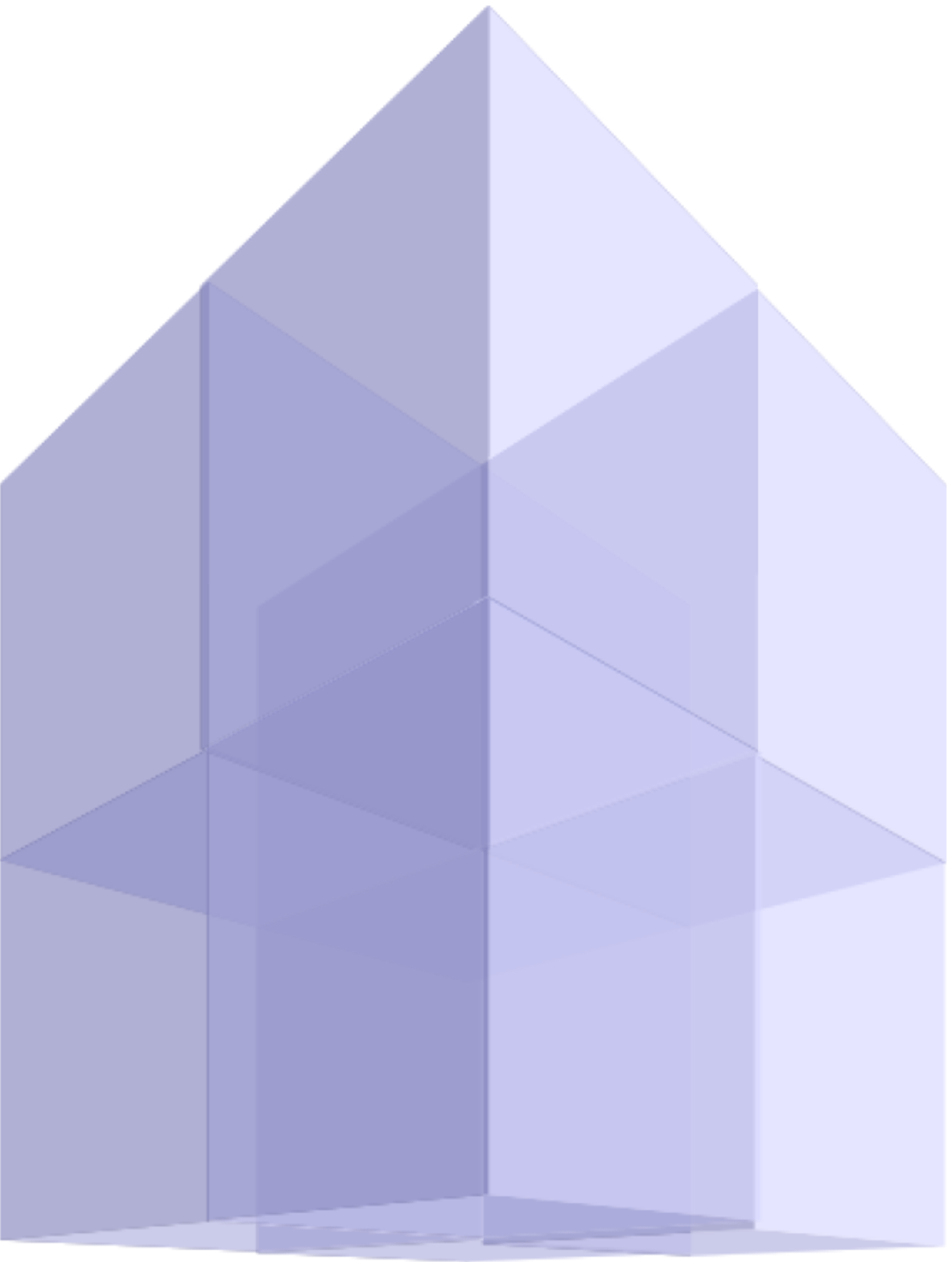,angle=0,width=.3\textwidth}}\quad
	\subfigure[Cube splitting $X\split W$]{\epsfig{figure=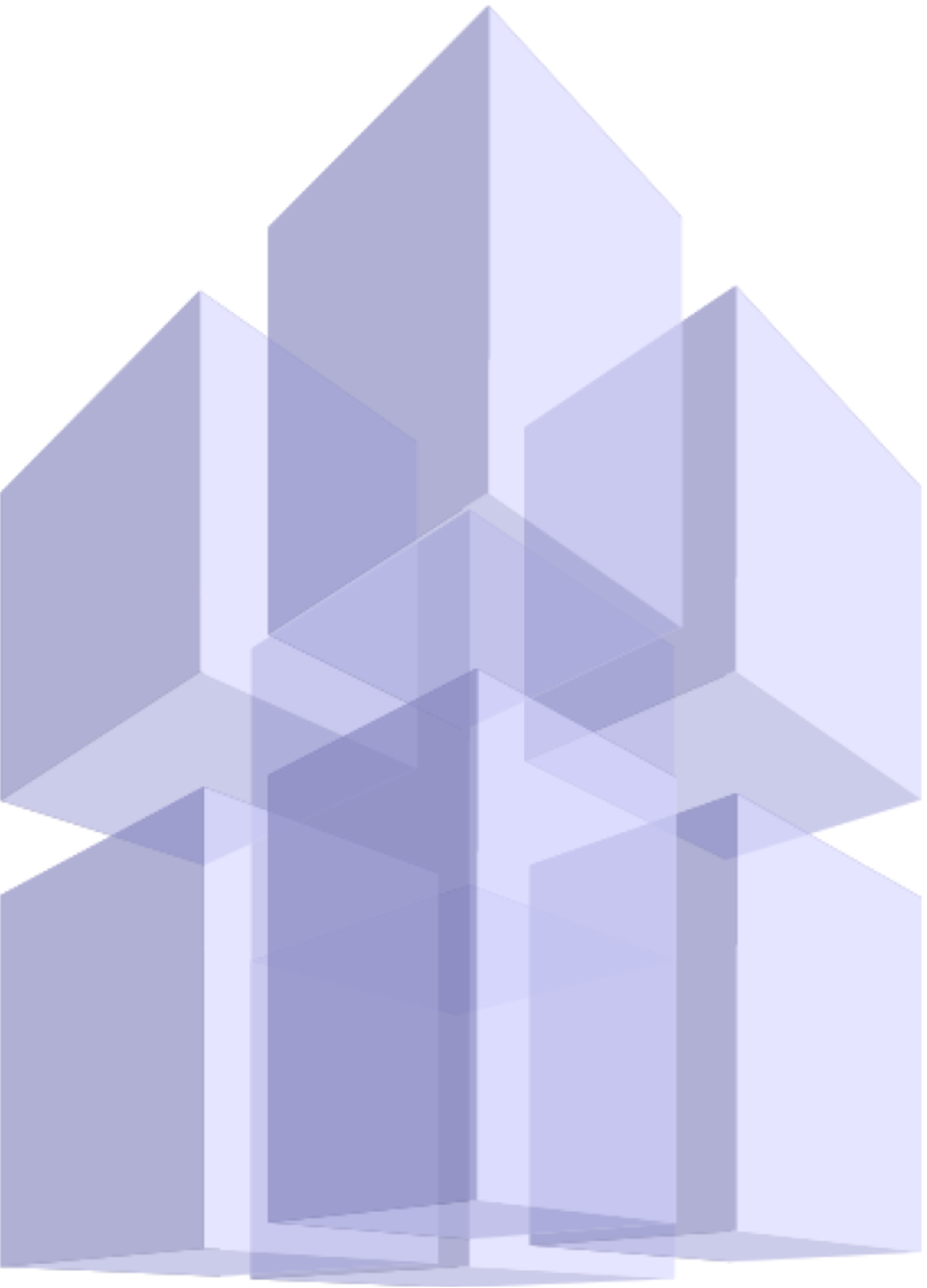,angle=0,width=.3\textwidth}}
	\caption{\label{cube} Subdividing and splitting a cube complex}
	\end{center}
\end{figure}

\begin{definition}
A {\it cubical polyhedron} $\mathcal{P}$ is a CAT(0) cube complex with a distinguished vertex $v \in \mathcal{P}$ which 
 is contained in every maximal cube. The polyhedron is determined by a simplicial graph $\Gamma$ which is the 1-skeleton of
$link(v)$, which is a flag CAT(1) spherical simplicial complex. To recover $\mathcal{P}$ from a connected simplicial graph $\Gamma$,
we may make a spherical
complex by taking a right-angled spherical $k$-simplex for each $k+1$ clique of $\Gamma$, glued
together by the natural inclusion of cliques. 
This gives a CAT(1) spherical complex since it is flag \cite[Theorem 5.18]{BridsonHaefliger}. Then attach the corner of a Euclidean
$k+1$ cube to each spherical $k$-simplex to get a cubical polyhedron. This is a non-positively
curved cube complex $P(\Gamma)$ \cite[Theorem 5.20]{BridsonHaefliger}. 
\end{definition}
The stars of vertices in an NPC cube complex are cubical polyhedra, and if we split $\dot{X}$ 
along all  of its hyperplanes, we get a union of stars of vertices and therefore cubical polyhedra.  

\begin{definition}
A cube complex with boundary pattern is a cube complex $X$ of bounded dimension together with locally convex subcomplexes
$\{\partial_1X,\ldots,\partial_n X\}$, $\partial_i X\subset X$ satisfying the following inductive definition (induct on the maximal dimension cube):
\begin{itemize}
\item
For each $i$, there is an isometrically embedded
open product neighborhood
$\partial_i X=\partial_i X\times 0 \subset \partial_i X\times [0,1) \subset X$. In particular, the dimension of each maximal
cube of $X_i$ is one less than the dimension of a cube of $X$ containing it. The intersection of $\partial_i X$ with
each cube must contain at most one face, and the intersection of all of the strata of the boundary pattern which
have non-trivial intersection with a fixed cube is a non-empty face. 
\item
For each $i$, the subcomplex $\partial_i X$  forms a 
cube complex with boundary pattern $\{\partial_j X\cap \partial_i X | j\neq i\}$, with induced collar neighborhoods $(\partial_j X\cap \partial_i X)\times[0,1) = (\partial_j X\times[0,1))\cap \partial_i X$. 
\end{itemize}
\end{definition}
What one may keep in mind for this definition is the analogy of a boundary pattern for
a hierarchy of a 3-manifold, arising in the work of Haken \cite{Haken62}. 

If $X$ is a cube complex with boundary pattern $\{\partial_1X,\ldots,\partial_n X\}$, then each $\partial_i X$ gets
a co-orientation of the collar neighborhood $\partial_i X\times [0,1)$, pointing into $X$ from $0$ to $1$ (into
the cube complex). We may similarly define a cubical orbihedron with boundary pattern as a cube complex
with boundary pattern quotient a group that preserves the boundary pattern.

{\bf Examples:} Take a graph $X$, and let $\partial_0 X\subset X$ be the vertices of $X$ which have degree 1, then $X$
 is a cube complex with boundary pattern $\partial_0 X$. 

Take a cubical polyhedron $P(\Gamma)$ associated to a simplicial graph $\Gamma$. For each vertex $v\in \Gamma$, 
consider the subcomplex defined by $link(v)\subset \Gamma$, $P(link(v))\subset P(\Gamma)$. 
Then $ [0,1]\times P(link(v))\subset P(\Gamma)$. The collection $\{ \{1\}\times P(link(v)) | v\in V(\Gamma)\}$
forms a boundary pattern of $P(\Gamma)$.  We will denote the union of this collection as $\partial P(\Gamma)$.
We will call components $\{1\}\times P(link(v))$ the {\it facets} of $P(\Gamma)$. If a collection of facets
of $P(\Gamma)$ have non-trivial intersection, then their intersection is a convex subcomplex we'll
call a {\it face}. The minimal faces will be points, which we will call {\it vertices of $P(\Gamma)$}. 

{\bf Remark:} The choice of terminology {\it cubical polyhedron} is meant to evoke a 
polyhedron. When $X$ is PL equivalent to a manifold, then each component of $X\split W$ is
homeomorphic to a ball, with the boundary pattern corresponding to the facets of the polyhedron.

\begin{definition}
Let $X$ be a cube complex with boundary pattern $\{\partial_1X,\ldots,\partial_n X\}$. 
Suppose there is an isometric involution $\tau: \partial_n X\to \partial_n X$ without fixed points, and
with the property that $\tau(\partial_i X\cap \partial_n X)=\partial_i X\cap \partial_n X$ for $i<n$. Then we may form the quotient
complex $X/\tau$, where for each cube  $c\subset \partial_n X$, we amalgamate the cubes $c\times[-1,0]$
and $\tau(c)\times [0,1]$ into a single cube isometric to $c\times [-1,1]$. We obtain an induced
boundary pattern $\{\partial_i X / (\tau_{| \partial_i X\cap \partial_n X})\ |  i<n\}$. This operation on $X$ is called {\it gluing a 
cube complex with boundary pattern}.
\end{definition}

\begin{definition} 
Let $X$ be a cube complex and let $W\subset X$ be a union of disjointly embedded 2-sided walls. 
We may {\it split} $X$ along $W$ by taking  the path-metric completion of $X-W$. This is
obtained from $\dot{X}-\dot{W}$ by adjoining two copies of $W$ on either side of the wall, which we'll
call $W^{\uparrow}$ and $W^{\downarrow}$, with co-orientations pointing into $X- W$. We will
denote this $X\split W$. If $X$ is a cube complex with boundary pattern $\{\partial_1 X,\ldots, \partial_m X\}$,
then $X\split W$ is a cube complex with boundary pattern $\{\partial_1 X, \ldots, \partial_m X, W^{\uparrow}\cup W^{\downarrow}\}$. 
This is the reverse operation from gluing a cube complex with boundary pattern. 
\end{definition}

If we have a cube complex $X$ with embedded walls $W\subset X$, then $X\split W$ (which really
means we split $\dot{X}$ successively along each component of $\dot{W}$) will
be a union of cubical polyhedra which are stars of vertices of $X$. 
For the complex $\mathcal{X}$ with walls $\mathcal{W}$ constructed at the end of Section \ref{quotient}, let $\mathcal{P(X)}$ be the set of cubical polyhedra which are
stars of vertices, and let $P_1, \ldots, P_p$ be orbit representatives under the action of $\mathcal{G}$ of the
cubical polyhedra of $\mathcal{X}\split \mathcal{W}$ which are vertex stars (we will think of these as the polyhedra
obtained by splitting $\mathcal{X}$ along its walls). Similarly, let $W_1,\ldots, W_w$ be orbit reps. of the walls $\mathcal{W}$
under the action of $\mathcal{G}$. Let $\mathcal{F(X)}$ denote the set of all
cubical polyhedra of the walls $\mathcal{W}$. These are the stars of midpoints of edges of $\mathcal{X}$ in $\mathcal{W}$. 
Let $\mathcal{F}=\{F_1, \ldots, F_f\}$ be
orbit representatives of the action of $\mathcal{G}$ on  $\mathcal{F(X)}$ (we will assume that each $F_i\subset W_j$ for some $j$).
 There is a canonical map $wall:\mathcal{F(X)}\to  V(\Gamma(\mathcal{X}))=\mathcal{W}$
defined by $wall(F)=W$ if $F\subset W\in V(\Gamma(\mathcal{X}))$.  Notice that there is a one-to-one correspondence
between $P_i$ and the vertices of $X/G$, and between $F_i$ and the edges of $X/G$.

\begin{definition}
Let $k=maxdegree(\Gamma(\mathcal{X}))$, then $C_{k+1}(\Gamma(\mathcal{X}))\neq \emptyset$. 
We want to define an equivalence relation $\simeq$ on $V(\Gamma(\mathcal{X}))\times C_{k+1}(\Gamma(\mathcal{X}))$. 
What this equivalence relation captures in part is how each wall is cut up by the previous walls in the ordering determined by a
wall coloring. In other words, a coloring determines a hierarchy for $\mathcal{X}$, and an induced hierarchy on 
each wall of $\mathcal{X}$. The equivalence relation captures how each wall is cut up by previous stages
of the hierarchy. This refinement is important for when we reconstruct the hierarchy to make sure after
gluing up the $j$th level of the hierarchy that the $j-1$st levels and lower are still matching up to finite index. 
We define it inductively. 

First, for $(v,c), (w,d) \in V(\Gamma(\mathcal{X}))\times C_{k+1}(\Gamma(\mathcal{X}))$, 
if $(v,c)\simeq (w,d)$, then we must have $v=w$ and $c(v)=d(w)$ (so the partition respects the vertex type). 
In other words, we want to define a partition refining the partition  $\{v\}\times C_{k+1}(\Gamma(\mathcal{X}))$ (but it 
will depend on the partitions associated to nearby vertices which is why we define it for all vertices simultaneously). 

\begin{enumerate}
\item We have $(v,c)\simeq (v,d)$ if $c(v)=d(v)=1$. 
\item We have $(v,c)\simeq (v,d)$ if $c(v)=d(v)=2$ and for all $w$ such that $(w,v)\in E(\Gamma(\mathcal{X}))$,
we have $c(w)=1 \iff d(w)=1$. 
\item[($j$)] The $j$th inductive step of the definition is given by: 
we have $(v,c)\simeq (v,d)$ if $c(v)=d(v)=j$ with $2\leq j \leq k+1$, and for all $w$ such that $(w,v)\in E(\Gamma(\mathcal{X}))$,
we have $(w,c)\simeq (w,d)$ if $c(w)<j$ or $d(w)<j$.
\end{enumerate}

Notice that the equivalence class of $(v,c)$ where $c(v)=j$ depends only on $c$ restricted to the ball of 
radius $j-1$ about $v$ in $\Gamma(\mathcal{X})$. This implies that the equivalence classes are
clopen sets as subsets of $\{v\}\times C_{k+1}(\Gamma(\mathcal{X}))$. In fact, if we think of the coloring $c$ as 
a Morse function on the vertices $V(\Gamma)$, then the equivalence class of $(v,c)$ depends only on the
``descending subgraph'' of $v$, consisting of the union of all paths in $\Gamma(\mathcal{X})$ starting at $v$ in which the values of $c$ are decreasing.

We now want to define an equivalence relation $\simeq$ on the set $\mathcal{F(X)}\times C_{k+1}(\Gamma(\mathcal{X}))$. 
We decree $(E,c)\simeq (E,d)$ if $(wall(E),c)\simeq (wall(E),d)$. 

We define an equivalence relation $\simeq$ on $\mathcal{P(X)}\times C_{k+1}(\Gamma(\mathcal{X}))$ 
$(P,c)\simeq (P,d)$ if for every facet $F\subset \partial P$, $(F,c)\simeq (F,d)$. In particular, the colors $c(F)$
of the facets $F\subset \partial P$ depend only on the $\simeq$ equivalence class of $(P,c)$. 

We have an action of $\mathcal{G}$ on each of these equivalence relations, by the action
for $g\in \mathcal{G}$ given by $g\cdot (v,c) = (g\cdot v, c\circ g^{-1})$, for $(v,c)\in \mathcal{W}\times C_{k+1}(\Gamma(\mathcal{X}))$,
and a similar formula for the action on faces and polyhedra. There are finitely many $\mathcal{G}$-orbits of equivalence
classes under the action of $\mathcal{G}$, and we may find representatives among $\{W_1,\ldots, W_w\} \times C_{k+1}(\Gamma(\mathcal{X}))$,
 $\{F_1,\ldots,F_f\} \times C_{k+1}(\Gamma(\mathcal{X}))$, and $\{P_1,\ldots,P_p\} \times C_{k+1}(\Gamma(\mathcal{X}))$. 
\end{definition}

\section{Gluing equations} \label{gluing equations}
We will consider weights on equivalence classes of polyhedra 
$\omega: \mathcal{P(X)}\times C_{k+1}(\Gamma(\mathcal{X}))/\simeq \to \mathbb{R}$ which are invariant
under the action of $\mathcal{G}$, so that $\omega(g\cdot(P,c))= \omega(P,c)$, for all $g\in \mathcal{G}$ and
satisfying the {\it polyhedral gluing equations}. A weight $\omega$ will be determined by its values on $[(P_j,c)]$, $1\leq j\leq p, c\in C_{k+1}(\Gamma(\mathcal{X}))$, and therefore is determined by finitely many variables. 
Given polyhedra $P, P'\subset \dot{\mathcal{X}}$ sharing a facet $F\subset \partial P, F\subset \partial P'$, we
get an equation on the weights for each equivalence class of $\{F\}\times C_{k+1}(\Gamma(\mathcal{X}))/\simeq$. 
For each equivalence class $[(F,c)]\in \{F\}\times C_{k+1}(\Gamma(\mathcal{X}))/\simeq$, we have the equation
$$\sum_{[(P,d)] | (F,d)\simeq(F,c)} \omega([(P,d)]) = \sum_{[(P',d)]| (F,d)\simeq(F,c)} \omega([(P',d)]).$$
The polyhedral gluing equations on the polyhedra equivalence class weights are 
the equations obtained for each equivalence class $[(F,c)]$. These equations are also $\mathcal{G}$-equivariant,
so in particular are determined by the equations for equivalence classes $[(F_i,c)]$, $1\leq i\leq f, c\in C_{k+1}(\Gamma(\mathcal{X}))$. 
Thus, we have finitely many equations determined by equivalence classes $[(F_i,c)], 1\leq i\leq f$ on finitely many variables
$\omega([P_j,c)]), 1\leq j\leq p$, together with the equations determined by $\mathcal{G}$-invariance.

For a measure $\mu\in M_{\mathcal{G}}(C_{k+1}(\Gamma(\mathcal{X})))$, we get non-negative polyhedral 
weights $\mu([(P,c)]) = \mu(\{d \in C_{k+1}(\Gamma(\mathcal{X})) |\ (P,c)\simeq (P,d)\})$ (and $\mu([F,c])$ is
similarly defined for each facet $F$). 
These weights satisfy the polyhedral gluing equations. Consider a facet $F = \partial P \cap \partial P'$, and
an equivalence class $[(F,c)]$ which defines a gluing equation.
Then using the additivity property of $\mu$, we have 
$$\sum_{[(P,d)] | (F,d)\simeq (F,c)} \mu([(P,d)])=\mu(\{d | (F,d)\simeq (F,c)\}) = $$ 
$$\mu([(F,c)]) =\sum_{[(P',d)] | (F,d)\simeq (F,c)} \mu([(P',d)]).$$
So $\mu$ gives a non-negative real solution to the polyhedral gluing equations. 

Since these equations are defined by finitely many linear equations with integral coefficients, there is
a non-negative non-zero integral weight function satisfying the polyhedral gluing equations, 
$\Omega: \mathcal{P(X)}\times C_{k+1}(\Gamma(\mathcal{X}))/\simeq\to \mathbb{Z}_{\geq 0}.$
In the next section we will use $\Omega$ to create a tower hierarchy which gives a finite-sheeted cover  of $X/G$ in $\mathcal{QVH}$.

\section{Virtually gluing up the hierarchy} \label{virtual gluing}

Let $(w,c)\in V(\Gamma(\mathcal{X}))\times C_{k+1}(\Gamma(\mathcal{X}))$. 
Let $\mathcal{W}_j^c =\cup \{ w\in V(\Gamma(\mathcal{X})) | c(w)=j\} \subset \mathcal{W}$, $1\leq j \leq k+1$, the union of
walls colored $j$ by $c$.
Suppose that $c(w)=j>1$, then define $w_1^c = w\split (w\cap \mathcal{W}_1^c)$. 
We may think of $w_1^c$ as being immersed in the wall $\dot{w}$. 
Then define inductively immersed complexes in $\dot{w}$ by $w_i^c = w_{i-1}^c\split (w^c_{i-1}\cap \mathcal{W}_i^c)$, for 
$2\leq i\leq j-1$. We don't split $w$ along $\mathcal{W}_j^c$ since $w\subset \mathcal{W}_j^c$. We will
use the notation $w_{j-1}^c = w^c$, since the $j=c(w)$ is implicitly determined (if $j=1$, then $w^c=w$). The complex
$w^c$ has a boundary pattern, given by $\partial_i(w^c)=w^c\cap \mathcal{W}_i^c$, $1\leq i\leq j-1$. 

{\bf Claim:} If $(w,c)\simeq (w,d)$, then $w^c=w^d$ (as cube complexes with boundary pattern).
In other words, $w^c$ depends only on the equivalence class $[(w,c)]$. This follows 
because $w^c$ is determined by $w\cap \mathcal{W}_i^c$, $1\leq i\leq j-1$, which
depends only on the equivalence class of $(w,c)$ since if $v$ is a component of $\mathcal{W}_i^c$
with $w\cap v\neq \emptyset$, then $(w,v)\in E(\Gamma(\mathcal{X}))$. 

Consider now the symmetries of $w^c$ which preserve the equivalence class. 
That is, consider $Stab(w^c)\leq  \mathcal{G}$, given by $g\in \mathcal{G}$ 
such that $g(w^c)=w^c$ (in particular, $g(w)=w$) and $(w,c\circ g^{-1})=(g(w),c\circ g^{-1})\simeq (w,c)$.  
Now define $w^c_{\mathcal{G}} = w^c/Stab(w^c)$, with its corresponding boundary pattern $\partial_i(w^c_{\mathcal{G}})=\partial_i(w^c)/Stab(w^c)$, $1\leq i\leq j-1$. 
In general, $w^c_{\mathcal{G}} $ will be an orbihedron with boundary pattern. 

For each $j$, $1\leq j\leq k+1$, let $\mathcal{Y}_j= \sqcup_{ [(w,c)], c(w)=j} w^c_{\mathcal{G}}$ (where we take precisely one $\mathcal{G}$-orbit representative 
of the equivalence relation $\simeq$ so that there are only finitely many equivalence classes $[(w,c)]$ up to the action of $\mathcal{G}$, and therefore $\mathcal{Y}_j$
is a compact cube complex). The orbi-complex $\mathcal{Y}_j$ has the property that for each $\mathcal{G}$-orbit of equivalence class $[(F,c)]$ with $c(F)=j$, 
there is a unique representative of $(F,c)$ in the complex $\mathcal{Y}_j$. 
At two extremes, we have $\mathcal{Y}_1= \cup \{W_1/Stab(W_1),\ldots, W_w/Stab(W_w)\}$, since the equivalence
class depends only on the orbit of the walls under the action of $\mathcal{G}$. We
have $\mathcal{Y}_{k+1} = \sqcup\{ [(F,c)]/Stab([F,c]) | c(wall(F))=k+1\}$, with boundary pattern $\partial_i F /Stab([F,c])= (F\cap \mathcal{W}_i^c)/Stab([F,c])$, $1\leq i\leq k$.  

\begin{proof}[proof of Theorem \ref{virtuallyspecial}] 
We will construct a sequence of (usually disconnected) finite cube complexes  $\mathcal{V}_j$, $k+1\geq j \geq 0$, with 
boundary pattern $\{\partial_1(\mathcal{V}_j), \ldots, \partial_{j}(\mathcal{V}_j)\}$ which have the following properties:
\begin{enumerate}
\item \label{map}
there is a locally convex combinatorial immersion $\nu_j: \mathcal{V}_j \to \dot{X}/G=\dot{\mathcal{X}}/\mathcal{G}$ 

\item \label{coloring}
$\mathcal{V}_j$ is glued together from copies of $\mathcal{G}$-orbits of  equivalence classes of  polyhedra $\mathcal{P(X)}\times C_{k+1}(\Gamma(\mathcal{X}))/\simeq$
in such a way that if polyhedra $P, P' \subset \mathcal{V}_j$ share a facet $F$, then the induced equivalence class of $F$
is the same. More formally, there is a decomposition of $\mathcal{V}_j$ into cubical polyhedra 
$\{P_h\}$, such that there is a lift $P_h\to \dot{\mathcal{X}}$ to a polyhedron of $\dot{\mathcal{X}}$ (well-defined
up to the action of $\mathcal{G}$)
which projects to the map $\nu_{j| P_h}: P_h \to \dot{\mathcal{X}}/\mathcal{G}$. 
Moreover, there is a coloring $c_h\in C_{k+1}(V(\Gamma(\mathcal{X})))$, with a well-defined equivalence class associated to the lift $P_h\to \dot{\mathcal{X}}$. 
If $P_g, P_h$ share a facet $F$, so that $F= \partial P_g\cap \partial P_h \subset \mathcal{V}_j$, then there is a lift $P_g\cup_F P_h \to \dot{\mathcal{X}}$ which
projects to the map $\nu_{j |}: P_g\cup_F P_h \to \dot{\mathcal{X}}/\mathcal{G}$.  We want the 
colorings to be compatible, in the sense that $(F, c_g)\simeq (F,c_h)$. Thus, there is a
well-defined map $c_j: \mathcal{F(V}_j) \to [k+1]$. 

\item \label{boundary}
The boundary
of $\mathcal{V}_j$ is the union of all facets $F$ contained in precisely one polyhedron $\partial P_g\subset \mathcal{V}_j$. Moreover,
the boundary pattern $\partial_i \mathcal{V}_j = \cup_{F \in \mathcal{F(V}_j), c_j(F)=i} F$, $1\leq i\leq j$. Thus, a facet $F$ is an  
interior facet  (contained in the boundary of two polyhedra) if and only if $c_j(F)>j$.

\item \label{involution}
The multiplicities of $\mathcal{G}$-orbits of equivalence classes of colored polyhedra making up $\mathcal{V}_j$ satisfy the polyhedral gluing equations. 
In particular, for each equivalence class $[(F,c)]$, $F=\partial P\cap \partial P'$, the number of lifts $P_g\to P$ with coloring $c_g$ which induce equivalent colorings
 $(F,c_g)\simeq (F,c)$ on $F$ is equal to the number of lifts $P_h\to P'$ which induce equivalent colorings $(F,c_h)\simeq (F,c)$. 
\end{enumerate}

The base case $\mathcal{V}_{k+1}$ is the collection of equivalence classes of polyhedra given by the solution to the
polyhedral gluing equations $\Omega$ found in the previous section. Recall we proved the existence of 
$\Omega: \mathcal{P(X)}\times C_{k+1}(\Gamma(\mathcal{X}))/\simeq \to \mathbb{Z}_{\geq 0}$ satisfying the
polyhedral gluing equations. For each equivalence class $[(P_j,c)]$, take $\Omega(P_j,c)$ copies of $P_j$, $1\leq j\leq p$,
keeping track of the coloring $c$ associated to each copy of $P_i$, and take the disjoint union of these to get $\mathcal{V}_{k+1}$. 
Each polyhedron has a locally convex map to $\dot{X}/G$, so condition (\ref{map}) holds. 
These have the empty gluing, each component of $\mathcal{V}_{k+1}$ has a
 lift to $\dot{\mathcal{X}}$ and coloring determined by the polyhedral equivalence class, so condition (\ref{coloring}) holds.
Every facet of $\mathcal{V}_{k+1}$ is a subset of $\partial \mathcal{V}_{k+1}$,
so there are no restrictions on the facets and condition (\ref{boundary}) holds.
Property (\ref{involution}) holds trivially since $\Omega$ is a solution to the polyhedral gluing equations.

Now, suppose we have constructed $\mathcal{V}_j$ with these properties, for $1\leq j\leq k+1$. Let's 
prove the existence of $\mathcal{V}_{j-1}$. The way that we will do this is to prove that $\partial_j\mathcal{V}_j$
covers components of $\mathcal{Y}_j$ with degree zero. By degree
zero, we mean that for each facet of $\mathcal{Y}_j$, the number
of facets of $\partial_j\mathcal{V}_j$ which cover the facet and
have one co-orientation is equal to the  number with the
opposite co-orientation, where the co-orientation points into 
the adjacent polyhedron. 
 Then we will appeal to Theorem \ref{virtualgluing}
to take a cover $\tilde{\mathcal{V}}_j$ of $\mathcal{V}_j$ which may be glued along $\partial_j\tilde{\mathcal{V}}_j$ to form $\mathcal{V}_{j-1}$. We must further
check that it satisfies the inductive hypotheses. 

{\bf Claim:} $\partial_j \mathcal{V}_j$ covers components of $\mathcal{Y}_j$ with degree zero. 

First, note that condition (\ref{map}) implies that each facet $F$ of $\mathcal{V}_j$ is contained in at most
two polyhedra of $\mathcal{V}_j$, because the map $\nu_j:\mathcal{V}_j\to \dot{X}/G$ is locally convex.
In particular, the map is injective on links of vertices lifted to $\dot{\mathcal{X}}$, and therefore is also injective on links of facets lifted to $\dot{\mathcal{X}}$. So
the gluing given in condition (\ref{coloring}) identifies facets of the polyhedra in pairs. As
described in condition (\ref{boundary}), the facets contained in exactly one polyhedron form
the boundary of $\mathcal{V}_j$, and therefore a facet of $\mathcal{V}_j$ which is not in the boundary of $\mathcal{V}_j$ must be
contained in precisely two polyhedra of $\mathcal{V}_j$. Also, because the map $\mathcal{V}_j\to \dot{X}/G$ is
locally convex, the link of each polyhedron vertex of $\mathcal{V}_j$ is the link of a product
of open intervals and half-open intervals. This implies that any path in $\partial_i\mathcal{V}_j$ may be deformed to lie in a sequence
of adjacent facets of $\partial_i\mathcal{V}_j$, meeting in codimension-one facets of $\partial_i\mathcal{V}_j$.   In fact, from the
inductive construction, $\mathcal{V}_j$ will have a hierarchy of length $k+1-j$ that induces such a hierarchy on each boundary
component as well. 

Consider a polyhedral facet $F$ involved in the boundary pattern $\partial_j\mathcal{V}_j$, 
which by hypothesis (\ref{coloring}) has a lift $F\to \dot{\mathcal{X}}$ and an associated equivalence class $[(F,c)]$, 
some $c\in C_{k+1}(\Gamma(\mathcal{X}))$. 
The adjacent polyhedron $\partial P \supset F$ has an equivalence class $[(P,d)]$ that is a polyhedron
of $\mathcal{V}_j$ by property (\ref{coloring}) such that $(F,d)\simeq (F,c)$. 
For a facet $F'$ of $\partial P$ adjacent to $F$ with color $d(F')>j$, there must be an adjacent polyhedron
$P' \subset \mathcal{V}_j$ containing $F'\subset \partial P'$, since this facet cannot occur as part of the boundary pattern of $\mathcal{V}_j$ by condition (\ref{boundary}).
Then  there is a unique facet $F''\subset \partial P'$ meeting $F'$ such that $F'\cap F=F''\cap F'$ and by condition (\ref{coloring}) 
 $F\cup_{F\cap F''} F'' \subset P\cup_{F'} P' \to \dot{\mathcal{X}}$
is a lift of the map $\nu_{j |}$ (from condition (\ref{map})) such that $F\cup F'' \subset wall(F)$ (so $wall(F)=wall(F'')$, see Figure \ref{develop}). Let $[(P',d')]$ be the equivalence class associated to $P'$ (which exists by condition (\ref{coloring}) ). 
Then $(F',d')\simeq (F',d)$ by the condition (\ref{coloring}). We have $(wall(F),wall(F'))\in E(\Gamma(\mathcal{X}))$. Also,
$d(wall(F)) < d(wall(F')), d'(wall(F''))< d'(wall(F'))$, by the inductive hypothesis on $\mathcal{V}_{j}$. 
Since $(F',d)\simeq(F',d')$, and therefore $(wall(F'),d)\simeq (wall(F'),d')$, we have $(wall(F),d) = (wall(F''),d) \simeq (wall(F''),d')$ by one of the 
conditions of the equivalence relation $\simeq$. Also, the lift $F\cup_{F\cap F''} F'' \to wall(F)^d \looparrowright wall(F)$, since
$d(F')>j$. 

\begin{figure}[htb] 
	\begin{center}
	\psfrag{F}{$F$}
	\psfrag{E}{$F'$}
	\psfrag{G}{$F''$}
	\psfrag{P}{$P$}
	\psfrag{Q}{$P'$}
	\psfrag{W}{$\mathcal{V}_j$}
	\psfrag{V}{$\partial_j\mathcal{V}_j$}
	\epsfig{file=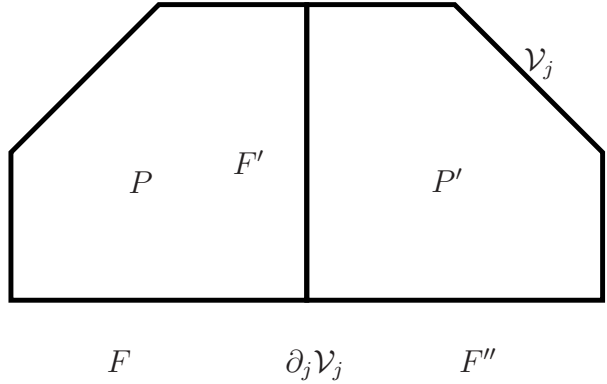, height=2in}
	\caption{\label{develop} Developing $\partial_j\mathcal{V}_j$ }
	\end{center}
\end{figure} 

Take a path  $\alpha:I\to \partial_j \mathcal{V}_j$ starting at $F$, and going through a sequence of facets $F=F_0, F_1,
F_2, \ldots, F_m$, such that $F_i$ is associated to a coloring $d_i$. We may assume each of these facets intersects its neighbors in codimension-one facets of $\partial_j\mathcal{V}_j$,
by the observation above. 
We see that once we choose a lift $F\to wall(F)^d\subset \mathcal{X}$, we get 
a lift $\tilde{\alpha}:I\to wall(F)^d$, and corresponding lifts $F_i\to wall(F)^d$. Moreover, $(wall(F),d_0)\simeq (wall(F),d_i)$. 
If $\alpha$ is a closed path so that $F_k=F$, then the lift $F_k\to wall(F)^d$ induces an equivalent coloring of $wall(F)$.
Thus, we see that the lift $F\to wall(F)^d$ is well-defined up to the action of $Stab(wall(F)^d)$, so we get a well-defined
lift of the component $Z$ of $\partial_j\mathcal{V}_j$ containing $F$ to a map $Z\to wall(F)^d/Stab(wall(F)^d)=wall(F)^d_{\mathcal{G}}$.

Conversely, if a facet $F' \subset \partial P$ adjacent to $F$ is colored $d(wall(F'))=i < j=d(wall(F))$, then $F'$ must be part of the
boundary pattern $\partial_i\mathcal{V}_j$ by condition (\ref{boundary}). Then $F'\cap F\subset \partial_i (\partial_j\mathcal{V}_j)$.
Thus, we have a  map $\pi: Z \to wall(Z)^c_{\mathcal{G}}$ which is a covering projection onto the component of its image.
The condition (\ref{involution}) ensures that the map $\partial_j\mathcal{V}_j\to \mathcal{Y}_j$ is degree zero, since for 
each facet equivalence class $[(F,c)]$ with $c(F)=j$, there is a unique representative of the $\mathcal{G}$-orbit of  $(F,c)$ in the complex
$\mathcal{Y}_j$. Thus, the number of representatives of $[(F,c)]$ in $\partial_j\mathcal{V}_j$ with one co-orientation will cancel with
the other co-orientation by the gluing equation for the class $[(F,c)]$. This finishes the proof of the claim that $\partial_j \mathcal{V}_j$ covers components of $\mathcal{Y}_j$ with degree zero.

Next, we need to show that $\partial_j\mathcal{V}_j$ is acylindrical in $\mathcal{V}_j$ in order to apply Theorem \ref{virtualgluing}. 
Suppose that there is an essential cylinder $(S^1\times [0,1], S^1\times\{0,1\})\to (\mathcal{V}_j, \partial_j\mathcal{V}_j)$. We may assume that
for each $z\in S^1$, $(z\times[0,1], z\times\{0,1\}) \to (\mathcal{V}_j, \partial_j\mathcal{V}_j)$ is a minimal length geodesic between the components
of $\partial_j\mathcal{V}_j$. 
There are elevations of each component of $\partial_j\mathcal{V}_j \to \dot{\mathcal{X}}$ which map to a locally convex immersion to a
wall  $\dot{W}\subset \dot{\mathcal{X}}$. We may therefore choose a compatible elevation  $(S^1\times [0,1],S^1\times\{0\}, S^1\times\{1\})\to (\dot{\mathcal{X}}, \dot{Y}_0, \dot{Y}_q)$, where $Y_0, Y_q\in V(\Gamma(\mathcal{X}))$, 
which must also be an essential cylinder. 
Therefore the walls $Y_0, Y_q\subset \mathcal{X}$ must be distance $\leq R$, and therefore $(Y_0, Y_q)\in E(\Gamma(\mathcal{X}))$, so
$Y_0$ and $Y_q$ must have
distinct colors in any coloring $c\in C_{k+1}(\Gamma(\mathcal{X}))$, so $c(Y_0)\neq c(Y_q)$. 
However, because there is a cylinder between the walls $Y_0, Y_q$, there must be a
sequence of walls $Y_0, Y_1, \ldots, Y_q$ such that the geodesic $z\times [0,1]$ intersects this
sequence of walls for some generic $z$. There will also be a sequence of facets $F_0, F_1, \ldots, F_q$,
$F_i\subset Y_i$ that the geodesic meets, and sequence of polyhedra $P_1, \ldots, P_q$,
with $F_{i-1}\cup F_i \subset \partial P_i, i=1,\ldots, q$.  
Associated to each $P_i$ is an equivalence class of colorings $[(P_i, d_i)]$, and since the facets 
$F_1, \ldots, F_{q-1}$ are interior to $\mathcal{V}_j$, we must have $(F_i, d_{i-1})\simeq (F_i, d_i)$.
In particular, $d_{i-1}(Y_0)=d_i(Y_0), d_{i-1}(Y_q)=d_i(Y_q)$, $i=1, \ldots, q$. 
But then $d_0(Y_0)=j, d_0(Y_q)=d_q(Y_q)=j$, which contradicts
the fact that $d_0(Y_0)\neq d_0(Y_q)$ since $(Y_0,Y_q)\in E(\Gamma(\mathcal{X}))$. 
Thus, we conclude that the cylinder does not exist, and therefore $\partial_j\mathcal{V}_j$ is acylindrical in $\mathcal{V}_j$.

To recap, we have an acylindrical subcomplex $\partial_j\mathcal{V}_j\subset \mathcal{V}_j$. Moreover, the components $Z$ of $\partial_j\mathcal{V}_j$
are partitioned into equivalence classes determined by the equivalence relation of the
equivalence class of $wall(Z)$ together with coloring. Each component covers a component of $wall(Z)^c_{\mathcal{G}}$ for
some $\simeq$ equivalence class $[(wall(Z),c)]$. Thus, there is a union of components  $Z_j\subseteq \mathcal{Y}_j$
such that there is a cover $\partial_j\mathcal{V}_j \to Z_j$. Moreover, the cover is degree $0$ with respect to the 
co-orientation. We split $\partial_j\mathcal{V}_j = \partial_j\mathcal{V}_j^{\uparrow} \sqcup \partial_j\mathcal{V}_j^{\downarrow} \sqcup \partial_j\mathcal{V}_j^{\circ}$, determined on 
each component by whether the cover of the corresponding component of $Z_j$ preserves
or reverses co-orientation, unless $Stab(wall(Z)^c)$ exchanges the sides of $wall(Z)$, in which case
we may ignore the orientation and it lies in $\partial_j\mathcal{V}_j^{\circ}$.

By Theorem \ref{virtualgluing}, there is a regular covering space $\tilde{\mathcal{V}}_j \to \mathcal{V}_j$, with
boundary pattern $\{ \partial_1\tilde{\mathcal{V}}_j, \ldots, \partial_j\tilde{\mathcal{V}}_j\}$ given by the preimages of $\partial_i\mathcal{V}_j$, 
such that the induced covering space $\partial_j\tilde{\mathcal{V}}_j \to Z_j$ is regular. Since the 
degree of the cover is zero, we must have that the covers $\partial_j\tilde{\mathcal{V}}_j^{\uparrow}\to Z_j$ and 
$\partial_j\tilde{\mathcal{V}}_j^{\downarrow}\to Z_j$
are common covers. After gluing the co-oriented components of $\partial_j\tilde{\mathcal{V}}_j$, we may take two copies of the
resulting complex, and glue the non-co-oriented components $\partial_j\tilde{\mathcal{V}}_j^{\circ}$ by co-orientation reversing isometries which exchange the sides in pairs (we'll rename the 2-fold cover $\tilde{\mathcal{V}}_j$
for simplicity). 
Thus, there is an isometric involution $\tau_j: \partial_j\tilde{\mathcal{V}}_j \leftrightarrow \partial_j\tilde{\mathcal{V}}_j$. 
We may form the quotient space $\mathcal{V}_{j-1}= \tilde{\mathcal{V}}_j / \tau_j$ by gluing the boundary pattern by $\tau_j$. We need to check that
the inductive hypotheses are satisfied for $\mathcal{V}_{j-1}$. 

Since the involution $\tau_j$ reverses co-orientation, we can see that the combinatorial immersion 
$\tilde{\mathcal{V}}_j \to \mathcal{V}_j \to \mathcal{\dot{X}/G}$ extends to an
immersion $\mathcal{V}_{j-1}\to \mathcal{\dot{X}/G}$. Moreover, since $\tau_j$ is
an involution of the boundary pattern, we see that $\mathcal{V}_{j-1}$ has locally
convex boundary since $\partial_j\tilde{\mathcal{V}}_j \subset \tilde{\mathcal{V}}_j$ has a collar neighborhood, 
and therefore the map to $\mathcal{\dot{X}/G}$ is locally convex, so condition (\ref{map}) is satisfied. 
The boundary pattern $\partial_i(\partial_j\mathcal{V}_j)$ is preserved by the involution $\tau_j$, since the coloring of the boundary
pattern is locally determined by the equivalence classes of walls being glued together for colors $i <j$. 
We define the boundary pattern of $\mathcal{V}_{j-1}$ by $\partial_i\mathcal{V}_{j-1}= \partial_i\tilde{\mathcal{V}}_j/\tau_{j|}$.
The interior facets will all have color $>j-1$, and the boundary facets will have color $\leq j-1$. So condition (\ref{boundary})
is satisfied. 
Since we have glued $\mathcal{V}_{j-1}$ out of copies of colored polyhedra in a way consistent with the gluing 
equations, and taking regular covers preserves the gluing equations, conditions (\ref{coloring}) and (\ref{involution})
are satisfied.

So all of the inductive hypotheses are satisfied. 

The complex $\mathcal{V}_0$ has trivial boundary pattern, and a locally convex map $\mathcal{V}_0\to \dot{X}/G$.
Therefore, this map is a finite-sheeted covering space. Moreover, by construction, $\mathcal{V}_0$ has
a quasi-convex hierarchy, so $\pi_1(\mathcal{V}_0)\in \mathcal{QVH}$ (in fact, the hierarchy is malnormal, so 
$\pi_1(\mathcal{V}_0)\in \mathcal{MQH}$). By \cite[Theorem 13.3 or Theorem 11.2]{Wise11} (see also Theorem \ref{QVH}),
$\mathcal{V}_0$ has a finite-sheeted special cover, and thus $X/G$ does. This finishes the proof of Theorem \ref{virtuallyspecial}. 
\end{proof}

\section{Conclusion}
Recall that a Haken 3-manifold is a compact irreducible orientable 3-manifold 
containing an embedded $\pi_1$-injective surface. 

\begin{theorem}[Virtual Haken conjecture \cite{Waldhausen}] \label{virtual Haken}
Let $M$ be a closed aspherical 3-manifold. Then there is a finite-sheeted
cover $\tilde{M}\to M$ such that $\tilde{M}$ is Haken.
\end{theorem}
\begin{theorem}[Virtual fibering conjecture, Question 18 \cite{Th:82}] \label{virtual fiber}
Let $M$ be a closed hyperbolic 3-manifold. Then there is a finite-sheeted
cover $\tilde{M}\to M$ such that $\tilde{M}$ fibers over the circle. Moreover,
$\pi_1(M)$ is LERF and large.
\end{theorem}
\begin{proof}[Proof of Theorems \ref{virtual Haken} and \ref{virtual fiber}]
From the geometrization theorem \cite{Per02, Per03, MorganTian08}, it is well-known that the virtual Haken
conjecture reduces to the case that $M$ is a closed hyperbolic 3-manifold. 
For a closed hyperbolic 3-manifold, we have the following result of Bergeron-Wise based on work of Kahn-Markovic \cite{KM09} (and
making use of seminal results of Sageev on cubulating groups containing codimension-one subgroups \cite{Sageev95}). 
\begin{theorem}\cite[Theorem 5.3]{BergeronWise09} \label{M cubulated}
Let $M$ be a closed hyperbolic $3$-manifold. Then $\pi_1M$ acts freely and
cocompactly on a CAT(0) cube complex.
\end{theorem}
Now, by Theorem \ref{virtuallyspecial}, $\pi_1(M)$ is virtually special. This 
implies that $\pi_1 M$ is LERF and large following from the virtual specialness by Cor. \ref{lerf}. 
Therefore $M$ is virtually Haken, and in fact $M$ is also virtually fibered by \cite[Corollary 14.3]{Wise11}.
\end{proof}
We also have the following corollary, resolving a question of Thurston.
\begin{corollary}[Question 15 \cite{Th:82}] 
Kleinian groups are LERF.
\end{corollary}
\begin{proof}
This follows combining \ref{virtual fiber} which proves that compact
hyperbolic 3-manifold groups are LERF, together with the implication
that therefore all finite-covolume Kleinian groups are LERF by \cite[Proposition 5.3]{MM-P}. 
It is well known that any Kleinian group embeds in a finite covolume Kleinian
group \cite{Mo}. 
\end{proof}

\appendix

\section{Filling virtually special subgroups}
\centerline{\textsc{by Ian Agol, Daniel Groves, and Jason Manning}}

This section will be devoted to proving the following theorem, which may be regarded
as a generalization of the main theorem of \cite{AGM08}, with the extra ingredient of the malnormal
virtually special quotient Theorem \ref{malnormal} \cite[Theorem 12.3]{Wise11}. 

\begin{theorem} \label{filling}
Let $G$ be a hyperbolic group, let $H \leq G$ be a quasi-convex virtually special subgroup. 
For any $g\in G-H$, there is a hyperbolic group $\mathcal{G}$ and a homomorphism
$\phi: G\to \mathcal{G}$ such that $\phi(g)\notin \phi(H)$ and $\phi(H)$ is finite. 
\end{theorem}

\begin{remark}
The conclusion of this theorem may be regarded as a weak version of 
subgroup separability. Under the hypotheses of the theorem, $H$ is subgroup separable in $G$
if one may also assume that the quotient group $\mathcal{G}$ is finite. 
\end{remark}

\begin{remark}
It ought to be possible to prove this result using the techniques to prove
\cite[Theorem 12.1]{Wise11}. However, we have decided to provide an alternative argument
which gives a geometric perspective on the notion of height, and uses hyperbolic Dehn filling
arguments from the literature instead of the small-cancellation theory
developed in \cite{Wise11}. 
\end{remark}

\begin{notation}
In this appendix, we will sometimes use the notation $A \dot{\leq} B$ to indicate
that $A$ is a finite-index subgroup of $B$.
\end{notation}

\begin{definition} \label{peripheral}
We define the malnormal core of $H$ and peripheral system induced by
$H$ on $G$.  Let $n$ be the height of $H$ in $G$. 
By \cite[Corollary 3.5]{AGM08}, there are finitely many $H$-conjugacy classes 
of minimal infinite subgroups of the form $H\cap H^{g_2} \cap H^{g_3}\cap \cdots\cap H^{g_j}$,
where $1\leq j\leq n$ and $\{g_1=1,g_2,\ldots, g_n\}$ are essentially
distinct, in the sense that $g_iH=g_jH$ if and only if $i=j$.

Choose one $H$-conjugacy class of each such subgroup in $H$, and replace it with its
commensurator in $H$ to obtain a collection of quasi-convex subgroups
$\mathcal{D}_0$ of $H$.
Eliminating
redundant entries which are $H$-conjugate, we obtain a collection $\mc{D}$,
which we will call the {\it malnormal core} of $H$ in $G$.   
The collection $\mathcal{D}$ gives
rise to a peripheral system of subgroups $\mathcal{P}$ in $G$ in two steps:
\begin{enumerate}
\item Change $\mathcal{D}$ to $\mathcal{D'}$ by replacing each $D\in
  D$ with $D' <G$ its commensurator in $G$.  
\item Eliminate redundant entries of $\mathcal{D'}$ to obtain $\mathcal{P}\subset \mathcal{D}'$
 which contains no two elements which are conjugate in $G$. 
\end{enumerate}

Call $\mathcal{P}$ the {\it peripheral structure on $G$ induced by
  $H$}. 
This peripheral structure is only 
well-defined up to replacement of some elements of $\mathcal{P}$ by conjugates. On the other
hand, replacing $H$ by a commensurable subgroup of $G$ does not affect the induced
peripheral structure. We consider two peripheral structures on a group to be the same
if the same group elements are parabolic (i.e. conjugate into
$\bigcup\mc{P}$) in the two structures.
\end{definition}
\begin{remark}
  Since $H$ is quasiconvex, so are the intersections $H\cap
  H^{g_1}\cap\cdots\cap H^{g_j}$ \cite[Lemma 2.7]{gmrs}.
  Because infinite quasiconvex subgroups of hyperbolic groups are finite index
  in their commensurators \cite[Lemma 2.9]{gmrs}, it follows that
  each of the elements of $\mc{D}$ or $\mc{P}$ contains some such $H\cap
  H^{g_1}\cap\cdots\cap H^{g_j}$ as a finite index subgroup.
\end{remark}
\begin{example}
If we have the minimal infinite subgroup $U=H\cap H^{g_1}\cap H^{g_2}$, 
then $U$ will appear as a subgroup of $H$ in three ways, up to $H$-conjugacy:
$U, U^{g_1^{-1}}=H^{g_1^{-1}} \cap H \cap H^{g_1^{-1}g_2}, U^{g_2^{-1}} = H^{g_2^{-1}}\cap H^{g_2^{-1}g_1}\cap H$.
Thus, there will be $\leq 3$ $H$--commensurators of conjugates of $U$
appearing in $\mc{D}$.  There could be strictly fewer: $U$ and $U^{g_1^{-1}}$ may
be commensurable in $H$ if $g_1\in \mathrm{Comm}_G(U)$.  Similarly, we get $\leq 3$
conjugates of the $G$--commensurator of $U$ in $G$ in $\mathcal{D}'$, 
but these are all $G$--conjugate, so only one element coming from $U$ remains in
$\mathcal{P}$.
\end{example}

We state also the malnormal virtually special quotient theorem \cite[Theorem 12.3]{Wise11} for reference.
\begin{theorem} \label{malnormal}
Let $G$ be a virtually special hyperbolic group. Let $\{H_1,\ldots,H_m\}$ be an almost malnormal
collection of quasiconvex subgroups. Then there exists finite-index subgroups $\dot{H}_i \dot{\unlhd} H_i$, $i=1,\ldots,m$ such
that for any further finite index subgroups $H'_i \dot{\leq}
\dot{H}_i$, the quotient $G/\ll H_1',\ldots, H_m'\gg$ is virtually
special. 
\end{theorem}
\begin{proof}[Proof of Theorem \ref{filling}]
Let $H\leq G$ be quasiconvex and virtually special, and let $g\in
G\setminus H$.  Let $h$ be the height of $H$ in $G$. 
We will induct on the height, noting that the height zero ($H$ finite)
case holds trivially.

Let $\mathcal{P}=\{P_1,\ldots,P_m\}$ be the peripheral system associated to $H \leq G$,
and $\mathcal{D}$ the peripheral system of $H$ 
from Definition \ref{peripheral}.  
 By Theorem
\ref{malnormal}, there are finite-index subgroups
$\dot{D}_j\dot{\unlhd} D_j$ 
for each $D_j\in \mathcal{D}$ such that for any further finite-index
subgroups $D'_j \dot{\leq} \dot{D}_j$, the quotient 
$H(D_1',\ldots,D_n'):=H/\ll \bigcup_j D'_j \gg $ is virtually
special. 

For each $D_j\in \mathcal{D}$, there is some unique $P_{i_j}$ and some $g_j$ so
that 
$$g_j^{-1} D_j g_j\ \dot{\leq}\  P_{i_j}.$$  The element $g_j$ is not unique, but
if $g_j'$ is another such element, then $g_j^{-1}g_j'\in P_{i_j}$.
In particular, different $G$--conjugates of $D_j$ in $P_{i_j}$ are
actually $P_{i_j}$--conjugates, so there are only finitely many of them.

Let $P_i\in \mc{P}$, and let 
\[ \mc{S}_i = \{ \dot{D}_j^g \mid D_j\in \mc{D}, g\in G, D_j^g\dot{\leq}
P_i\}.\]
By the way $\mc{D}$ and $\mc{P}$ are
defined, $\mc{S}_i$ is never empty.  By
the argument in the last paragraph, $\mc{S}_i$ is a finite collection,
so $I_i:= \bigcap \mc{S}_i \dot{\unlhd} P_i$.

Theorems \ref{oldstuff} and \ref{heightdecreases}
imply that there is a finite subset $B\subset
\bigcup\mc{P}$ so that whenever $\phi\co G\to G(N_1,\ldots,N_m)$ is an
$H$--filling (see Definition \ref{H-filling})
satisfying $(\bigcup N_i)\cap B=\emptyset$
and $N_i\dot{\lhd} P_i$, then:
\begin{enumerate}
\item  The image $\phi(H)$ is quasiconvex of height $<h$ in the hyperbolic group
  $G(N_1,\ldots,N_m)$. (Theorem \ref{heightdecreases})
\item 
  $\phi(H)\cong H(K_1,\ldots,K_n)$, where $H(K_1,\ldots,K_n)$ is the
  induced filling of $H$, described in Remark \ref{rem:induced filling}.  
  (Theorem \ref{oldstuff}.\eqref{eq:induced injective})
\item  $\phi(g)\notin\phi(H)$. (Theorem \ref{oldstuff}.\eqref{separateg})
\end{enumerate}

Since $H$ is residually finite, and each $P_i$ is a finite extension
of a subgroup of $H$, each $P_i$ is residually finite.  Hence there
are normal subgroups $N_i\dot{\unlhd} P_i$ so that $(\bigcup
N_i)\cap B = \emptyset$.  These normal subgroups need
not define an $H$--filling, but we can instead consider the subgroups
\[ N_i' = N_i\cap I_i. \]
Then $\phi\co G\to G(N_1',\ldots,N_m')$ is an $H$--filling inducing
a filling $H\to H(K_1,\ldots,K_n)$ satisfying the hypotheses of
Theorem \ref{malnormal}.  In particular, the image $\bar{H}$ of $H$ in
$\bar{G}:=G(N_1',\ldots,N_m')$ is virtually special.  By Theorem
\ref{heightdecreases}, $\bar{G}$ is hyperbolic and $\bar{H}\leq
\bar{G}$ is quasiconvex, of height $<h$.
Moreover, Theorem \ref{oldstuff} implies $\phi(g)\notin \phi(H)$.

By induction, there is a quotient $\bar{\phi}\co\bar{G}\to \mc{G}$
so that $\bar{\phi}(\phi(g))\notin \bar{\phi}(\phi(H))$ and
$\bar{\phi}(\phi(H))$ is finite.
\end{proof}

\subsection{Definitions}

\begin{definition} (See \cite[Section 3]{AGM08})
Let $G$ be a hyperbolic group and $H$ a quasi-convex subgroup, and let
$\mc{P}$ and $\mc{D}$ be the induced peripheral structures on $G$ and
$H$ described above.  Let $X$ be the cusped space of $(G,\mc{P})$ and
$Y$ the cusped space of $(H,\mc{D})$ (with respect to choices of
generating sets).  The inclusion $\phi\co H\to G$ sends peripheral subgroups in $\mc{D}$ into (conjugates of) peripheral subgroups in $\mc{P}$, and so induces a proper $H$-equivariant Lipschitz map $\check{\phi} \co Y \to X$.
We say that $(H,\mc{D})$ is {\em $C$-relatively quasiconvex} in $(G,\mc{P})$ if
$\check{\phi}$ is $C$-Lipschitz and
has $C$-quasiconvex image in $X$. 
\end{definition}
In \cite[Appendix A]{MM-P} it is explained that the above definition agrees with other notions of relative quasiconvexity, such as those in \cite{Hruska}.

The following is proved in \cite{AGM08} under the assumption that $G$ is
torsion-free.  It was extended to the general setting in \cite{mp2}.
\begin{proposition}\cite[Proposition 3.12]{AGM08},\cite[Corollary 1.9]{mp2}
The pairs $(H,\mc{D})$ and $(G,\mc{P})$ are both relatively hyperbolic and with these
peripheral structures $(H,\mc{D})$ is a relatively quasi-convex subgroup of $(G,\mc{P})$.
\end{proposition}

\begin{definition}\label{H-filling}
  Let $(H,\mc{D})$ be a relatively quasi-convex subgroup
  of $(G,\mc{P})$, where $\mc{P}=\{P_1,\ldots,P_m\}$.  
  Let $\{N_i\lhd P_i\}$ be given.  The quotient 
  \[ G(N_1,\ldots,N_m) := G/\ll N_1\cup\cdots\cup N_m\gg \]
  is a \emph{filling of $(G,\mc{P})$}.  
  It is an \emph{$H$--filling} if $N_i^g\subset P_i^g\cap H$ whenever $H\cap
  P_i^g$ is infinite.
\end{definition}
\begin{remark}
  The current definition of $H$--filling agrees with the one in
  \cite{AGM08} only in case $G$ is torsion-free.  As explained in
  \cite[Appendix B]{MM-P}, Definition \ref{H-filling} is the correct
  extension in case there is torsion.
\end{remark}

\begin{remark} \label{rem:induced filling}
As explained in \cite[Definition 3.2]{AGM08}, an $H$-filling $G(N_1,\ldots , N_m)$ induces a filling $H(K_1,\ldots , K_n)$ of $H$:  For each $D_i \in \mc{D}$, there is a $c_i \in G$ and $P_{j_i} \in \mc{P}$ so that $D_i \subseteq c_i P_{j_i} c_i^{-1}$.  Then $K_i = c_iN_{j_i}c_i^{-1} \cap D_i$.  The inclusion
$H \into G$ induces a homomorphism $H(K_1,\ldots , K_n) \to G(N_1,\ldots ,N_m)$.
\end{remark}

\begin{definition}
  Let $(G,\mc{P})$ be a relatively hyperbolic group.
  We say that a statement $\mathsf{S}$ about fillings $G(N_1,\ldots,N_m)$
  holds \emph{for all sufficiently long fillings} if there is a finite
  set $B\subset \bigcup\mc{P}$ 
  so that whenever $G(N_1,\ldots,N_m)$ is a filling so that
  $\bigcup_{i=1}^n N_i$ does not contain $B$, then $\mathsf{S}$
  holds.

  Similarly, if $(H,\mc{D})$ is a
  relatively quasiconvex subgroup of $(G,\mc{P})$, a statement
  $\mathsf{S}$ holds \emph{for all sufficiently long $H$--fillings}
  if  there is a finite set $B\subset\bigcup\mc{P}$ 
  so the statement $\mathsf{S}$ holds for all $H$--fillings
  $G(N_1,\ldots,N_m)$ so that $\bigcup_{i=1}^n N_i$ does not contain
  $B$.
\end{definition}

Obviously if $\mathsf{S}$ holds for all sufficiently long fillings,
then $\mathsf{S}$ holds for all sufficiently long $H$--fillings.
The fundamental theorem of relatively hyperbolic Dehn filling can 
be stated:
\begin{theorem}\label{rhdf}\cite{Osin07,DGO} (cf. \cite{rhds} in the
  torsion-free case)  Let $G$ be a group and
  $\mc{P}=\{P_1,\ldots,P_m\}$ a collection of subgroups so that
  $(G,\mc{P})$ is relatively hyperbolic, and
  let $F\subset G$ be finite.  Then for all sufficiently long
  fillings $\phi\co G\to \bar{G}:=G(N_1,\ldots,N_m)$,
  \begin{enumerate}
  \item $\ker(\phi|_{P_i})=N_i$ for each $P_i\in\mc{P}$;
  \item $(\bar{G},\{\phi(P_1),\ldots,\phi(P_m)\})$ is relatively hyperbolic;
    and
  \item $\phi|_F$ is injective.
  \end{enumerate}
\end{theorem}

Our chief new Dehn filling result in this appendix is the following:

\begin{theorem}\label{heightdecreases}
  Let $G$ be hyperbolic, and let $H$ be height $k\geq 1$ and  quasi-convex in $G$.  Suppose
  that $\mc{D}$ and $\mc{P}=\{P_1,\ldots,P_m\}$ are as in Definition \ref{peripheral}. 
  Then for all sufficiently long $H$--fillings 
  \[ \phi: G \to \bar{G}:=G(N_1,\ldots,N_m) \]
  with $N_i\lhd P_i$ finite index for all $i$, 
  the subgroup $\phi(H)$ is quasi-convex of height
  strictly less than $k$ in the hyperbolic group $\bar{G}$.
\end{theorem}

\begin{remark}
We proved Theorem \ref{heightdecreases} in \cite{AGM08} under the
assumption that $G$ was torsion-free.  Much of the proof from
\cite{AGM08} still works without that assumption, but our argument that
height is reduced in the quotient depended on the machinery of Part 2 of
\cite{rhds}, in which torsion-freeness is assumed.  Our main
innovation in this appendix is a completely different proof that
height decreases under Dehn filling.
\end{remark}

\subsection{Geometric finiteness}

Geometric finiteness is a dynamical condition.  We recall the relevant
definitions.
\begin{definition}
  Let $M$ be a compact metrizable space with at least $3$ points, and
  let $\Theta(M)$ be the 
  set of unordered distinct triples of points in $M$.  Any action of
  $G$ on $M$ induces an action on $\Theta(M)$.  The action of $G$ on
  $M$ is said to be a \emph{convergence group action} if the induced
  action on $\Theta(M)$ is properly discontinuous.
\end{definition}

\begin{definition}
  Suppose $G\curvearrowright M$ is a convergence group action.  A
  point $p\in M$ is a \emph{conical limit point} if there is a
  sequence $\{g_i\}_{i\in \bN}$ and a pair of points $a$, $b$ so that
  $g_i p\to b$ but for every $x\in M\setminus\{p\}$, we have $g_i x\to
  a$. 

  A point $p$ is \emph{parabolic} if $\Stab_G(p)$ is infinite but
  there is no infinite order $g\in G$ and $q\neq p\in M$ so that
  $\Fix(g)=\{p,q\}$.

  A parabolic point $p$ is called \emph{bounded parabolic}
 if $\Stab_G(p)$ acts cocompactly on $M\setminus \{p\}$.
\end{definition}

\begin{definition} \label{d:geom fin}
  The action $G\curvearrowright M$ is \emph{geometrically finite} if
  every point in $M$ is a conical limit point or a bounded parabolic
  point.  Say that $(G,\mc{P})$ \emph{acts geometrically finitely} on
  $M$ if all of the following hold:
  \begin{enumerate}
  \item $G\curvearrowright M$ is a geometrically finite convergence action.
  \item Each $P\in \mc{P}$ is equal to $\Stab_G(p)$ for some bounded
    parabolic point $p$.
  \item For any bounded parabolic point $p$, the stabilizer
    $\Stab_G(p)$ is conjugate to exactly one element of $\mc{P}$.
  \end{enumerate}

  Let $X$ be a $\delta$--hyperbolic $G$--space, so that $(G,\mc{P})$
  acts geometrically finitely on $\partial X$.  Then we say that
  $(G,\mc{P})$ acts \emph{geometrically finitely on $X$}.  
\end{definition}

It is useful when talking about Dehn filling to allow parabolic
subgroups to be finite.  We will use the following definitions:
\begin{definition}
  Let $G\curvearrowright M$ be a convergence action, and say that
  $p\in M$ is a \emph{finite parabolic point} if $p$ is isolated and
  has finite stabilizer.

  For $\mc{P}$ a finite collection of subgroups of $G$, write
  $\mc{P}_\infty$ for the subcollection of infinite subgroups, and
  $\mc{P}_f$ for the subcollection of finite subgroups.
  Suppose $G$ acts on the
  compact metrizable space $M$.
  Let $M'$ be obtained from $M$ by removing all isolated  points. 
  Say that $(G,\mc{P})$ acts \emph{weakly geometrically finitely} on
  $M$ (or that the action is \emph{WGF}) if all of the following occur:
  \begin{enumerate}
  \item   $(G,\mc{P}_\infty)$ acts geometrically finitely on
    $M'$.
  \item Each $P\in \mc{P}_f$ is equal to $\Stab_G(p)$ for some $p\in
    M\setminus M'$.
  \item Every $p\in M\setminus M'$ is a finite parabolic point, with
    stabilizer conjugate to exactly one element of $\mc{P}_f$.
  \end{enumerate}

  Finally, if $X$ is a $\delta$--hyperbolic $G$--space, then we say
  that the action of $(G,\mc{P})$ on $X$ is \emph{WGF} whenever the
  action of $(G,\mc{P})$ on $\partial X$ is WGF.
\end{definition}

\begin{proposition}
  Let $(G,\mc{P})$ be relatively hyperbolic.  Then the action of
  $(G,\mc{P})$ on its cusped space is WGF.

  Conversely, if $(G,\mc{P})$ has a WGF action on a space $M$, then
  $(G,\mc{P})$ is relatively hyperbolic.
\end{proposition}
\begin{proof}
  For $\mc{P}=\mc{P}_\infty$, a proof of the equivalence can be found
  in \cite{Hruska}.

  The pair $(G,\mc{P})$ is relatively hyperbolic if and only if
  $(G,\mc{P}_\infty)$ is relatively hyperbolic.
  Indeed, for any generating set $S$ of $G$, the cusped space
  $X_\infty=X(G,\mc{P}_\infty,S)$ quasi-isometrically embeds into
  $X=X(G,\mc{P},S)$.  The complement $X\setminus X_\infty$ is composed
  of combinatorial horoballs based on finite graphs.  Thus
  $X$ is quasi-isometric to $X_\infty$ with $\#(\mc{P}_f)$ rays
  attached to vertex of the Cayley graph of $G$.

  Suppose that $(G,\mc{P})$ is relatively hyperbolic, so that
  $X$ is Gromov hyperbolic.  Then $X_\infty$ is also Gromov
  hyperbolic, and $\partial X'$ can be canonically identified with
  $\partial X_\infty$.    Thus $G$ acts geometrically finitely on
  $\partial X'$.
  Moreover, the isolated points of $\partial X$
  are in one to one correspondence with the left cosets of elements of
  $\mc{P}_f$; the point corresponding to $tP$ has (finite) stabilizer
  equal to $tPt^{-1}$.  Thus $(G,\mc{P})$ acts weakly geometrically
  finitely on $X$.

  Conversely, if $(G,\mc{P})$ has a WGF action on $M$, then
  $(G,\mc{P}_\infty)$ has a geometrically finite action on $M'$, so
  $(G,\mc{P}_\infty)$ is relatively hyperbolic.  Since
  $\mc{P}\setminus\mc{P}_\infty$ is composed of finite subgroups of
  $G$, the pair $(G,\mc{P})$ is also relatively hyperbolic.
\end{proof}
\begin{remark}
  Given that $(G,\mc{P})$ is relatively hyperbolic if and only if
  $(G,\mc{P}_\infty)$ is relatively hyperbolic, it is often convenient
  to simply ignore the possibility of finite parabolics, as for
  example in \cite{Hruska}.  In the present setting it is important to
  keep track of them, as otherwise we would not get uniform control of
  the geometry of cusped spaces of quotients, as in Theorem
  \ref{oldstuff} below.
\end{remark}

Suppose that $X$ and $Y$ are $\delta$-hyperbolic $G$-spaces.  A
$G$-equivariant quasi-isometry from $X$ to $Y$ induces a
$G$-equivariant homeomorphism from $\partial X$ to $\partial Y$.
Since the property of being a (weakly) geometrically finite action on a $\delta$-hyperbolic space is defined in terms of the boundary, we have the following result.

\begin{lemma}
Suppose that $(G,\mc{P})$ admits a WGF action on a $\delta$-hyperbolic space $X$ (as in Definition \ref{d:geom fin}) and that $f \co X \to Y$ is a $G$-equivariant quasi-isometry to another $\delta$-hyperbolic $G$-space.  Then the action of $G$ on $Y$ is WGF.
\end{lemma}

\begin{remark}\label{free cusped space}
  In the presence of $2$--torsion, the action of $G$ on the cusped space
  $X(G,\mc{P},S)$ may not be free, though it is always free on the
  vertex set.
  In what follows, it is convenient to replace the graph
  $X(G,\mc{P},S)$ as defined in \cite{rhds} with a graph having the
  same vertex set, but on which $G$ acts freely.  Since $G$ acts
  freely on the vertex set already, we can modify $X(G,\mc{P},S)$ to a
  graph with a free $G$--action by replacing each edge by two edges,
  corresponding to the two choices of orienting the edge.  This does
  not change the coarse geometry of $G$ or any of the statements we
  apply from \cite{rhds,AGM08,MM-P,Hruska}.  
  From now on, when
  we refer to the cusped space, we will assume it has been
  modified as just explained to make the action free.
\end{remark}

\subsection{Height from multiplicity}

\begin{definition} \label{d:fqc}
  $(H,\mc{D})<(G,\mc{P})$ is \emph{fully quasiconvex} if it is
  relatively quasiconvex and whenever $gDg^{-1}\cap P$ is infinite,
  for $D\in \mc{D}, P\in \mc{P}$, then $[P:gDg^{-1}]<\infty$.
\end{definition}
In this section, $(G,\mc{P})$ is relatively hyperbolic, and
$(H,\mc{D})$ is a fully quasiconvex subgroup.  We allow the
possibility that $\mc{P}$ and $\mc{D}$ are empty.

If $\mc{P}$ (and therefore $\mc{D}$) is empty, we take $\Gamma$ to be
any graph on which $G$ acts freely and cocompactly, and choose
$\tilde{*}$ to be some arbitrary vertex.
Otherwise, we take $\Gamma$ to be the $1$--skeleton of a cusped
space $X(G,\mc{P},S)$, modified to have a free $G$--action as in
Remark \ref{free cusped space}.
\footnote{Below, when applying Theorem \ref{oldstuff} to a quotient of $\Gamma$ by $G$, we will consider a graph which is the cusped space with some extra loops attached (in an equivariant way) to some vertices.  It is straightforward to check that our arguments work as written for this slightly different space.  In fact, with only a little extra work, one can take $\Gamma$ to be any graph with a free WGF $G$-action, but we decided to stick with the more restrictive setting in the interests of brevity.}
In this case $\Gamma$ contains a Cayley graph for $G$, and we take
$\tilde{*}=1\in G\subset\Gamma$. 

\begin{definition} \label{d:R-hull}
Let $R\geq 0$.  An {\em $R$-hull for $H$ acting on $\Gamma$} is a connected
$H$-invariant sub-graph
$\tilde{Z}\subset \Gamma$ so that all of the following hold.
\begin{enumerate}
\item\label{contains_basepoint} $\tilde{*}\in \tilde{Z}$.
\item\label{contains_geodesics} If $\gamma$ is a geodesic in $\Gamma$ with endpoints in the
  limit set of $H$, then the $R$--neighborhood of $\gamma$ is
  contained in $\tilde{Z}$.
\item\label{containsB} If $\mc{P}$ is nonempty and $B$ is a horoball of $\Gamma$ whose
  stabilizer in $H$ is infinite, then $B'\subset \tilde{Z}$ where $B'$
  is some horoball nested in $B$.
\item\label{WGF} The action of $(H,\mc{D})$ on $\tilde{Z}$ is WGF.
\end{enumerate}
\end{definition}

Let $\tilde{Z}$ be an $R$-hull for $H$ acting on $\Gamma$, let $Z =
\tilde{Z} / H$ be the quotient of $\tilde{Z}$ by the $H$--action, and
let $Y = \Gamma / G$ be the quotient of $\Gamma$ by the $G$--action.  
If we let $*_H\in Z$ and $*\in Y$ be the images of $\tilde{*}$,
we obtain canonical surjections
$s\co\pi_1(Z,*_H)\to H$ and $s\co \pi_1(Y,*)\to G$.
Moreover the canonical  map
\[ i: Z\to Y \]
which is the composition of the inclusion $Z \hookrightarrow \Gamma / H$ with the quotient map $\Gamma / H \to Y = \Gamma / G$
agrees with the inclusion of $H$ into $G$.

\begin{definition}\label{Sn}
Let $n>0$, and define the following subset of $Z^n$:
\begin{equation}
  S_n = \{(z_1,\ldots,z_n)\mid i(z_1)=\cdots=i(z_n)\}\setminus\Delta
\end{equation}
where $\Delta=\{(z_1,\ldots,z_n)\mid z_i=z_j\mbox{ for some }i\neq j\}$ is the ``fat diagonal'' of $Z^n$.  
Let  $s:\pi_1(Z,*_H)\to H$ be the canonical surjection.  
Let $\varpi_1,\ldots,\varpi_n$ be the $n$
projections of $S_n$ to $Z$.
\end{definition}

(In Stallings' language \cite{Stallings83}, $S_n$ is that part of the \emph{pullback} of
$n$ copies of $i\co Y\to Z$ which lies outside $\Delta$.)

Let $C$ be a component of $S_n$, with a choice of basepoint
$p=(p_1,\ldots,p_n)$.  For $i \in \{ 1, \ldots , n \}$ define maps 
$\tau_{i,C} \co \pi_1(C) \to H$ as follows:

Choose a maximal tree $T$ in $Z$.  For each vertex $v$ of $Z$, the
tree gives a canonical path $\sigma_v$ from $*_H$ to $v$, allowing the
fundamental groups of $Z$ at different basepoints to be identified.
To simplify notation define $\sigma_i = \sigma_{p_i}$.
Now, the map $\varpi_i \co C \to Z$ induces a well-defined map
$(\varpi_i)_* \co \pi_1(C,p) \to \pi_1(Z,*_H)$, taking a loop $\gamma$ based
at $p\in C$ to the loop $\sigma_i\gamma\bar\sigma_i$.
We define $\tau_{C,i} = s \circ (\varpi_i)_* \co \pi_1(C,p) \to H$.  
Since $H$ acts on $\tilde{Z}$ by covering translations, the map $s$ can be seen by lifting paths starting and finishing at the basepoint in the usual way.  Once we've used the path in the maximal tree to make based loops in $C$ map to paths in $Z$ starting and finishing at the basepoint $*_H$, the same is true of the maps $\tau_{C,i}$.

\begin{definition}\label{d:multiplicity}
  The \emph{multiplicity} of $Z\to Y$ is the largest $n$ so that $S_n$
  contains a component $C$ so that for all $i \in 1, \ldots , n$ the group
  \[	\tau_{C,i}\left(\pi_1(C)\right)	\]
  is an infinite subgroup of $H$.
\end{definition}

\begin{lemma}\label{conjugateimages}
For a fixed component $C$ of $S_n$, the groups
\[A_i=\tau_{C,i}(\pi_1(C,p))<H\] are conjugate in $G$.  Specifically, if
$\sigma_i$ are defined as above, and $g_{i,j}$ is represented by the
loop $i\circ\sigma_i\cdot i\circ\bar\sigma_j$, then 
$g_{i,j}A_j g_{i,j}^{-1}=A_i$.
\end{lemma}
\begin{proof}
As in the above discussion, the basepoint of $C$ is
$p=(p_1,\ldots,p_n)$, and for each $i$ there is a canonical
path $\sigma_i$ in $T\subset Z$ connecting the basepoint $*_H$ of $Z$
to $p_i$.  We also recall the map $i\co Z\to Y$ takes $*_H$ to $*$
and induces the inclusion $H<G$ in the sense that the diagram
\cd{ \pi_1(Z,*_H)\ar[r]^i\ar[d]^s & \pi_1(Y,*)\ar[d]^s \\ H \ar[r] & G }
commutes, where the vertical arrows are the canonical surjections.

Let $q = i(p_1)=\cdots=i(p_n)$.
The paths $i\circ\sigma_i$ all begin at $*$ and end at $q$, so any
concatenation of two of them gives a loop in $Y$ representing
an element $g_{i,j}$
of $G$ conjugating one of the images of $\pi_1(C,p)$ to another.
Precisely, for $i,j\in \{1,\ldots,n\}$ we get an element
$g_{i,j}$ represented by $i\circ\sigma_i\cdot i\circ\bar\sigma_j$ so that
\begin{equation}
  g_{i,j}\tau_{C,j}(\alpha)g_{i,j}^{-1}=\tau_{C,i}(\alpha), \forall\alpha\in\pi_1(C,p).
\end{equation}
\end{proof}

We aim in this section for the following:

\begin{theorem}\label{heightmultip}
  Let $R$ be bigger than the quasi-geodesic stability constant
  $D=D(\delta)$ specified in the proof below.
  With the above notation, the height of $H$ in $G$ is equal to the
  multiplicity of $Z\to Y$.
\end{theorem}
\begin{remark} 
It is instructive to contemplate the proof of this
theorem when $G$ is a Kleinian group, and $H$ a geometrically
finite subgroup. Then it is not hard to verify that the multiplicity
of a convex core of $H$ is equal to the height of $H$. In fact,
the arguments in this section are motivated by carrying this
geometric argument over to the broader category of hyperbolic
groups. 
\end{remark}

Before doing the proof, we state and prove a corollary. 
\begin{corollary}\cite{gmrs}
  The height of a quasiconvex subgroup of a hyperbolic group is finite.
\end{corollary}
\begin{proof}
  Suppose $H$ is quasiconvex in $G$. Let $\Gamma$ be
  a Cayley graph for $G$, so that $H\subset\Gamma$ is
  $\lambda$--quasiconvex.  It is easy to see that
  $\tilde{Z}=N_{\lambda+5\delta+R}(H)$ is an $R$--hull for $H$.  Since
  $G$ (resp. $H$) acts cocompactly on $\Gamma$ (resp. $\tilde{Z}$),
  the complexes $Z$ and $Y$ are both finite.  Thus $S_n$ is empty for
  large $n$.
\end{proof}

\begin{proof}[Proof of Theorem \ref{heightmultip}]
We first bound multiplicity from below by height, and then conversely.
\medskip

\noindent
  \textbf{(multiplicity $\geq$ height):} Suppose that $H$ has height $\geq n$.  There are
  then $(H,g_2 H,\ldots, g_n H)$ all distinct so that 
  $J=H\cap H^{g_2}\cap\cdots \cap H^{g_n}$ is infinite.
  Since $(G,\mc{P})$ is relatively hyperbolic, every infinite subgroup
  of $G$
  either contains a hyperbolic element or is conjugate into some $P\in
  \mc{P}$.  This follows immediately from the classification of isometries of $\delta$-hyperbolic spaces (see \cite[Section 8.2, p. 211]{gromov:wordhyperbolic}) and from the definition of WGF action. 
  The proof therefore
  breaks up naturally into these two cases.
  \begin{case}
    The intersection $J$ contains a
    hyperbolic element $a$.
  \end{case}
  By replacing $a$ by a power we may suppose that $a$ has a
  $(K,C)$--quasi-geodesic axis $\tilde\gamma_a$, where $K$ and $C$ depend only on
  $\delta$.  Quasi-geodesic stability implies that $\tilde\gamma_a$ lies
  Hausdorff distance at most $D$ from a geodesic, where $D$ depends
  only on $\delta$.  So as long as $R>D$, the geodesic $\tilde\gamma_a$ lies
  in $\tilde{Z}\cap g_2\tilde{Z}\cap\cdots\cap g_n\tilde{Z}$.
  Let $\pi_Z$ be the natural projection from $\tilde{Z}$ to $Z$.
  For $t\in \bR$, define $\gamma_a: \bR\to Z^n$ as follows:
  \[ \gamma_a(t) = \left(\pi_Z( \tilde\gamma_a(t) ), 
                         \pi_Z( g_2^{-1}\tilde\gamma_a(t)),\ldots,
                         \pi_Z( g_n^{-1}\tilde\gamma_a(t))\right) \]
  Since $G$ acts freely and the $g_i$ are essentially distinct,
  $\gamma_a$ misses the diagonal.  Since its coordinates differ only
  by elements of $g$, $\gamma_a$ has image in $S_n$.  Moreover
  projection of $\gamma_a$ to any component gives a loop of infinite
  order in $H$. Thus we've shown that a component of $S_n$ has an 
  element with infinite order projection to $G$, and therefore the multiplicity
  of $H$ is $\geq n$. 
  \begin{case}
    The intersection $J$ is
    conjugate into $P\in \mc{P}$.
  \end{case}
  In this case $J$ preserves some horoball $B$ of $\Gamma$.
  By point \eqref{containsB} in the definition of $R$--hull, there is
  a horoball $B'$ nested inside $B$ so that $B'\subset \tilde{Z}$.  By
  possibly replacing $B'$ with a horoball nested further inside, we have 
  \[B'\subset \tilde{Z}\cap g_2\tilde{Z}\cap\cdots\cap g_n\tilde{Z}.\]
  It follows that 
  \[A = \left\{(\pi_Z(b),\pi_Z(g_2^{-1}(b)),\ldots,\pi_Z(g_n^{-1}(b)))\mid
  b\in B'
   \right\}\]
  lies in some component $C$ of  $S_n$.  Moreover, each
  $\tau_{C,i}(\pi_1(A))<\tau_{C,i}(\pi_1(C))$ is conjugate to $J$, hence infinite.

\medskip
\noindent
  \textbf{(height $\geq$ multiplicity):}  
  Suppose the multiplicity of $Z\to Y$ is $n$.  Let $C\subset S_n$ be
  a component with infinite fundamental group, and let
  $p=(p_1,\ldots,p_n)\in C$.  We define the paths $\sigma_i$ 
  from $*_H$ to $p_i$ as in the discussion before Definition
  \ref{d:multiplicity}.  Recall the homomorphisms
  \[ \tau_{C,i}\co \pi_1(C,p)\to H < G \]
  are defined by
  $\tau_{C,i}([\gamma])=[i\circ(\sigma_i\cdot\gamma_i\cdot\bar\sigma_i)]$,
  for any loop $\gamma=(\gamma_1,\ldots,\gamma_n)$ based at $p$ in
  $C$. According to Lemma \ref{conjugateimages}, if we let $A_i =
  \tau_{C,i}(\pi_1(C,p))$, and $g_{i,j}=[i\circ\sigma_i\cdot
    i\circ\bar\sigma_j]$,
  then
  \[ A_j^{g_{i,j}}=A_i. \]
  In particular, writing $g_i = g_{1,i}$, we have
  \[ H\cap H^{g_2}\cdots \cap H^{g_n} \supseteq A_1 \]
  is infinite.  To establish the height of $H$ is at least $n$, we
  need to show that $(1,g_2,\ldots,g_n)$ are essentially
  distinct. 
 
  Let $\tilde{T}$ be the lift of $T$ to $\Gamma$ which
  includes the point $\tilde{*}$, and let
  $\gamma=(\gamma_1,\ldots,\gamma_i)$ be a loop in $C$ 
  based at $p$.  
  For each $i$, the path $\sigma_i$ has a unique lift to $\tilde{T}$.
  Let $\tilde{\gamma_i}$
  be the unique lift of $\gamma_i$ starting at the terminus of
  $\tilde{\sigma_i}$.
  Then $g_i(\gamma_i)=\gamma_1$.

  Since $p\in C$ lies outside the fat diagonal of $Z^n$, the paths
  $\gamma_1,\ldots,\gamma_n$ are all distinct.  In particular, the
  lifts $\tilde{\gamma_1},\ldots,\tilde{\gamma_n}$ are also distinct.

  Suppose that $(1,g_2,\ldots,g_n)$ are not essentially distinct.
  Then 
  we would have (writing $g_1=1$) $g_j = g_i h$ for some $1\leq
  i<j\leq n$ and some $h\in H$.  But then
  $\tilde\gamma_j=h^{-1}\tilde\gamma_i$.  Projecting back to $Z$ we
  have $\gamma_j = \gamma_i$,  contradicting the fact that
  $\gamma$ misses the fat diagonal of $Z^n$.
\end{proof}

\subsection{Height decreases}

Again, we have the setup: 
$(H,\mc{D})$ is fully quasiconvex in $(G,\mc{P})$.  We now assume that
$\mc{D}$ and $\mc{P}$ are nonempty, and that $\Gamma = X(G,\mc{P},S)$
for some finite generating set $S$ for $G$.
The graph $\Gamma$ is acted on weakly geometrically finitely by $(G,\mc{P})$. 
Moreover given any finite 
generating set $T$ for $H$, there is a $\lambda>0$ and
an $H$--equivariant, proper, $\lambda$--Lipschitz
map (defined in \cite[Section 3]{AGM08})
\[ \check\iota: X(H,\mc{D},T)\to \Gamma \]
$X(H,\mc{D},T)$ with $\lambda$--quasiconvex image.  Indeed, the
existence of such a map is the \emph{definition} of relative
quasi-convexity in \cite{AGM08}.  This definition is shown to be
equivalent to the usual ones in \cite[Appendix A]{MM-P}.

Recall that the cusped space is built by attaching combinatorial
horoballs to a Cayley graph, so there are canonical inclusions
$H\hookrightarrow X(H,\mc{D},T)$ and $G\hookrightarrow \Gamma$.
The map $\check\iota$ extends the natural inclusion map of $H$ into $G$.
\begin{lemma}
  There is some $N$ so that the $N$--neighborhood of
  the image of $\check\iota$ is a $0$--hull for the action of $H$ on
  $\Gamma$. 
\end{lemma}
\begin{proof}
There are four conditions to check.  We will prove that each of them hold for any large enough value of $N$, and then take the maximum of the four lower bounds.

For a number $D \ge 0$ and a subset $A \subset \Gamma$, let $\mc{N}_D(A)$ denote the closed $R$-neighborhood of $A$ in $\Gamma$.  Let $Y_D = \mc{N}_D\left(\check\iota\left(X(H,\mc{D},T)\right) \right)$.

Condition \eqref{contains_basepoint}:
Since $\tilde{*} = 1 \in H \subset G$, we have $\tilde{*} \in Y_D$ for any $D \ge 0$.

Condition \eqref{contains_geodesics}:
Suppose that $\xi_1, \xi_2 \in \Lambda H$, the limit set of $H$ in $\partial \Gamma$, and suppose that $l$ is a geodesic between $\xi_1$ and $\xi_2$.  To satisfy the second condition of Definition \ref{d:R-hull} (with $R = 0$), we need $l$ to be contained in $Y_D$ for large enough $D$.
The points $\xi_1$ and $\xi_2$ are limits of elements of $H$.  Since $Y_0$ is $\lambda$-quasi-convex, a geodesic between any two elements of $H$ is contained in $Y_\lambda$.  It is now straightforward to see that $l$ is contained in $Y_{\lambda + 2\delta}$.

Condition \eqref{containsB}:
Suppose that $B$ is a horoball of $\Gamma$ whose stabilizer in $H$ is infinite.  Algebraically, this gives peripheral subgroups $D$ of $H$ and $P$ of $G$, and $g \in G$ so that $gDg^{-1} \cap P$ is infinite.  
The condition from Definition \ref{d:fqc} ensures that $\left[ P \co gDg^{-1} \right] < \infty$.  This implies that there is some $D_0$ so that $B \subset Y_{D_0}$.  Since there are only finitely many such $D, P$ and $g$, up to the action of $G$, the number $D_0$ may be taken to work for all such horoballs  $B$.

Condition \eqref{WGF}:
The final condition from Definition \ref{d:R-hull} is that $H$ acts weakly 
geometrically finitely on $Y_N$.  This is true for any $N > 0$.  
Note that $Y_N$ is quasi-convex and quasi-isometric to $Y_0$, the
image of $\check\iota$.
Because the peripheral subgroups
of $H$ are finite index in maximal parabolic subgroups of $G$, the map
$\check\iota$ is a quasi-isometric embedding.  (The proof is similar
to the proof that a quasiconvex subgroup of a hyperbolic group is
quasi-isometrically embedded.)  The map $\check\iota$ therefore gives
an $H$-equivariant quasi-isometry between $X(H,\mc{D},T)$ and $Y_N$.
Since $(H,\mc{D})$ acts weakly geometrically finitely on $X(H,\mc{D},T)$, it
follows that $(H,\mc{D})$ acts weakly geometrically finitely on $Y_N$.
\end{proof}

\begin{definition}
  Let $\tilde{Z}_0$ be the $N$--neighborhood of $\mathrm{Im}(\check\iota)$ for $N$
  sufficiently large that $\tilde{Z}_0$ is a $0$--hull.  For $R>0$, let $\tilde{Z}_R$
  be the $R$--neighborhood of $\tilde{Z}_0$.  Clearly $\tilde{Z}_R$ is
  an $R$--hull.  Let $Z_R$ be the quotient of $\tilde{Z}_R$ by the
  $H$--action, and let $Y$ be the quotient of $\Gamma$ by the
  $G$--action, as in the previous section.
\end{definition}

\begin{theorem}\label{oldstuff}
Let $G$ be hyperbolic, $H<G$ quasiconvex, and let $g\in G\setminus H$.
Let $(G,\mc{P})$, $(H,\mc{D})$ be the relatively hyperbolic structures
from Definition \ref{peripheral}.   Let $A$ be a finite set in $G$.

Then for all sufficiently long $H$--fillings $\phi\co G\to
G(N_1,\ldots,N_m)$:
\begin{enumerate}
\item If $K=\ker(\phi)$, then
  $\bar{\Gamma}:=\Gamma / K$ is $\delta'$--hyperbolic for some
  $\delta'$ independent of the filling, and is (except for trivial
  loops) equal to the cusped space for $(\bar{G},\bar{\mc{P}})$.
  In particular $(\bar{G},\bar{\mc{P}})$ is relatively hyperbolic.
\item Let $q\co \Gamma\to \bar{\Gamma}$.  For some $\lambda'$ independent of the filling,
  $\mathrm{Im}(q\circ\check\iota)$ is $\lambda'$--quasiconvex.   Thus
  the induced filling $(\bar{H},\bar{\mc{D}})$ is relatively quasiconvex in
  $(\bar{G},\bar{\mc{P}})$.
  \item \label{eq:induced injective} The induced map $H(K_1,\ldots , K_n) \to G(N_1,\ldots , N_m)$ (described in Remark \ref{rem:induced filling}) is injective.
  \item\label{separateg} $\phi(g)\notin\phi(H)$.
  \item\label{eq:injective on finite} $\phi|_A$ is injective.
  \end{enumerate}
\end{theorem}
\begin{proof}
It is straightforward to see that $\Gamma / K$ is almost the cusped
space for $(\bar{G},\bar{\mc{P}})$ as advertised, and we leave the
details to the reader.  The extra loops come from horizontal edges in
horoballs between elements in the same $K$--orbit.
When $G$ is torsion-free, the fact that $\delta'$ is independent
of the filling is \cite[Proposition 2.3]{AGM08}.  If one quotes the filling
theorem from \cite{Osin07} (which holds in the presence of torsion) instead of \cite{rhds},
then the proof from \cite{AGM08} works as written.

Again when $G$ is torsion-free, that $(\bar{H},\bar{\mc{D}})$ is $\lambda'$-quasiconvex
for $\lambda'$ independent of the filling is \cite[Proposition 4.3]{AGM08} (with a slightly
different definition of $H$-filling, as discussed above).  That this works in the presence of
torsion is explained in \cite[Appendix B]{MM-P}, and the addition of the loops to the cusped space for $(\bar{G},\bar{\mc{P}})$ as above does not affect this quasiconvexity.

Further (again when $G$ is torsion-free), \cite[Proposition 4.4]{AGM08} says that the induced map
$H(K_1,\ldots , K_n) \to G(N_1,\ldots , N_m)$ is injective.  Once we have noted (as in \cite[Appendix B]{MM-P}) that \cite[Lemma 4.2]{AGM08} holds in the presence of torsion and with the amended definition of $H$-filling, the proof of \cite[Proposition 4.4]{AGM08} works as written.

Now, \cite[Proposition 4.5]{AGM08} implies in the torsion-free case
that $\phi(g)\notin \phi(H)$ for sufficiently large fillings.  Again,
the extension of this proposition in the presence of torsion is
explained in \cite[Appendix B]{MM-P}. 

Finally, \eqref{eq:injective on finite} is part of \cite[Theorem 1.1]{Osin07}.
\end{proof}

\begin{proposition}\label{uniformR}
  For any $R'>0$, there is an $R$ so that $q(\tilde{Z}_{R})$ is an $R'$--hull
  for the action of $\bar{H}$ on $\bar{\Gamma}$, for all sufficiently
  long fillings.
  (In particular, $R$ does not depend
  on the choice of long filling.)
\end{proposition}
\begin{proof}
  Let $\gamma$ be a geodesic joining limit points of $\bar{H}$.  The
  $\lambda'$--neighborhood of $W:=q(\mathrm{Im}(\check\iota))$ contains
  $\gamma$, by quasi-convexity.  Thus the $R$--neighborhood of
  $\gamma$ is contained in the $R+\lambda'$--neighborhood of $W$,
  hence in the image of $Z_{R+\lambda'}$.

  The other conditions follow from the relative quasiconvexity of
  $\bar{H}$ in $\bar{G}$.
\end{proof}

Let $G\to \bar{G}$ be a sufficiently long filling to satisfy the
conclusions of Theorem \ref{oldstuff}, so that $\bar{\Gamma}$ is
$\delta'$--hyperbolic, $\mathrm{Im}(q\circ\check\iota)$ is
$\lambda'$--quasiconvex, and so on.

Fix $R'$ bigger than the constant $D(\delta')$ from Theorem
\ref{heightmultip}.  
Then $q(\tilde{Z}_R)$ detects height in any
sufficiently large filling, in a sense which we will describe below.

\begin{lemma}\label{hullquotientembeds}
  For all sufficiently long fillings $\phi\co G\to G(N_1,\ldots,N_m)$,
  if $K = \ker(\phi)$, $K_H = K\cap H$ and $k\in K\setminus K_H$, then 
  $k \tilde{Z}_R\cap\tilde{Z}_R=\emptyset$.
\end{lemma}
\begin{proof}
  The set $A = \{g\in G\mid g\tilde{Z}_R\cap\tilde{Z}_R\neq \emptyset\}$
  is a finite union of left cosets of $H$,
  \[ A = \bigsqcup_{i=0}^l g_i H\mbox{, } g_0 = 1.\]
  Applying Theorem \ref{oldstuff} for $g=g_1,\ldots,g=g_l$, we
  conclude that for all sufficiently long fillings, 
  $\phi(g_i)\notin\phi(H)$ for $i>0$.  Equivalently
  $g_ih\notin K$ for any $h\in H$, and any $i>0$.  Thus for $k\in
  K\setminus K_H$, we have $k\notin A$, and so
  $k\tilde{Z}_R\cap\tilde{Z}_R=\emptyset$. 
\end{proof}

Let $\tilde{\bar{Z}}_R$ be the quotient of $\tilde{Z}_R$ by $K_H$, and let $\bar{Z}_R=Z_R$ be the quotient of $\tilde{\bar{Z}}_R$ by
the action of $H$.  By Lemma \ref{hullquotientembeds},
$\tilde{\bar{Z}}_R$ embeds in $\bar{\Gamma}$.
Now we have a commutative diagram,
\cd{
\tilde{Z}_R \ar@(ul,u)[]^H\ar[r] \ar[d] & \Gamma \ar@(u,ur)[]^G\ar[d] \\
\tilde{\bar{Z}}_R\ar[r]\ar@(l,ul)[]^{H/K_H} & \bar{\Gamma}\ar@(ur,r)^{G/K},
}
where the horizontal maps are inclusions and the vertical maps are
quotients by $K_H$ and $K$ respectively.
After taking quotients by the relevant groups we get the diagram,
\cdlabel{ZvbarZ}{
Z_R \ar[d] \ar[r]^i & Y \ar[d]\\
\bar{Z}_R \ar[r]^{\bar\imath} & \bar{Y}}
where the vertical maps are homeomorphisms, and the horizontal maps
are immersions inducing the inclusions $H\to G$ and $\bar{H}\to
\bar{G}$.  As maps, $i$ and $\bar\imath$ are exactly the same, so the sets
$S_n$ and $\bar{S}_n$ are the same.  For each $i \in \{ 1 ,\ldots , n\}$
and each component $C$ of $S_n$ we have maps 
\[	\tau_{C,i} \co \pi_1(C) \to H	,	\]
and
\[	\bar\tau_{C,i} \co \pi_1(C) \to \bar{H}	.	\]
Since the quotient $Z_R = \tilde{Z}_R / H$ can also be thought of as
$\tilde{\bar{Z}}_R / (H/K_H)$,
we see that the homomorphisms
$\bar\tau_{i,C}$ all factor as $\bar\tau_{i,C} = q|_H \circ \tau_{i,C}$.

In particular, if $\gamma$ is a loop
in $\bar{S}_n$ so that $\bar\tau_{C,i}\left( [\gamma ] \right)$ is infinite for each $i \in \{ 1 , \ldots , n \}$ 
then it must be that $\tau_{C,i} \left( [\gamma ] \right)$ is already infinite for each $i$.  Therefore
we have the following result.

\begin{corollary}
  The height of $\bar{H}$ in $\bar{G}$ is at most the height of $H$ in $G$.
\end{corollary}

We now specialize to the case that $(H,\mc{D})< (G,\mc{P})$ comes from
a quasiconvex subgroup $H$ of a hyperbolic group $G$, so that $\mc{D}$
is the malnormal core of $H$ and $\mc{P}$ the induced peripheral
structure on $G$.
\begin{theorem}\label{heightgoesdown}
Assume $\mc{P}$ is the peripheral structure induced on $G$ by the
quasiconvex subgroup $H$, and let $G\to G(N_1,\ldots,N_m)$ be a
sufficiently long $H$--filling.
In case every filling kernel $N_i$ has finite index in $P_i$,
the height of $\bar{H}$ in $\bar{G}$ is strictly less than that of
$H$ in $G$.
\end{theorem}
\begin{proof}
  Suppose that $H$ has height $n$ in $G$ and that, contrary to the
  conclusion, $\bar{H}$ has height $n$ in $\bar{G}$.  
  
Fix $R' > 0$.    By Proposition
  \ref{uniformR}, there is an $R$ so that for any long enough filling
  the set $q(\tilde{Z}_R )$ is an $R'$-hull for the action of $\bar{H}$
  in $\bar{G}$.  We choose $R'$ large enough so that it satisfies the hypotheses
  of Theorem \ref{heightmultip}.  Specifically, we make sure $R' >
  D(\delta')$ for the universal constant of hyperbolicity $\delta'$
  from Theorem \ref{oldstuff}.

  By Theorem
  \ref{heightmultip}, the multiplicity of the map $\bar\iota \co \bar{Z}_R \to \bar{Y}$
  is $n$.  Let $\bar{C}$ be a component of $\bar{S}_n$ (with basepoint $\bar{p} = \left( \bar{p}_1, \ldots , \bar{p}_n \right)$) so that each of the subgroups $\bar{A}_i = \bar{\tau}_{\bar{C},i} \left( \pi_1\left(\bar{C},\bar{p}\right) \right)$ are infinite.\footnote{We add a bar to our notation in the obvious way in the quotient.}
  By Lemma \ref{conjugateimages}, the groups $\bar{A}_i$ are all conjugate in $\bar{H}$, and so there are $\bar{g_{i,j}} \in \bar{G}$
  so that $\bar{g}_{i,j}\bar{A}_j\bar{g}_{i,j}^{-1} = \bar{A}_i$.  Since $G$ is a hyperbolic group, any infinite subgroup contains an infinite
  order element, so let $\bar{a} \in \bar{A}_1$ be such an infinite order element, and suppose that $\gamma_{\bar{a}}$ is a loop in $\bar{C}$ based at $\bar{p}$ so that $\bar{\tau}_{i,j} \left( [ \gamma_{\bar{a}} ] \right)$.
  
  Now, consider the diagram \eqref{ZvbarZ}.  The vertical maps are
  homeomorphisms, and so (as in the above discussion), induce a
  homeomorphism between $\bar{S}_n$ and $S_n$.  Let $C$ be the component of $S_n$ corresponding to $\bar{C}$, let $p$ be the associated basepoint, and let $\gamma_a$ be the loop in $C$ associated to $\gamma_{\bar{a}}$.  As in the discussion above,
  the image of each $\tau_{C,i} \left( [\gamma_a] \right)$ is infinite in $H$.  This shows that $a = \tau_{C,1} \left( [\gamma_a ] \right)$
  is an element of infinite order in the intersection in $G$ of $n$
  essentially distinct conjugate of $H$.  Thus $a$ lies in 
  a conjugate of an element of $\mc{D}$.  Since the filling kernels $N_i$ have finite index in $P_i$, some power of $a$ is contained in
  the kernel of the filling map $G \to \bar{G}$.  But the image of $a$ in $\bar{G}$ is clearly $\bar{a}$, which shows that $\bar{a}$
  cannot have infinite order, contrary to assumption.  This completes the proof. 
\end{proof}

We now prove Theorem \ref{heightdecreases}.
\begin{proof}[Proof of Theorem \ref{heightdecreases}]
  We have $G$ hyperbolic, $H<G$ quasiconvex and height $k\geq 1$.  We
  then have $(H,\mc{D})$ relatively quasiconvex in $(G,\mc{P})$ where
  $\mc{D}$ is the malnormal core of $H$, and
  $\mc{P}=\{P_1,\ldots,P_m\}$ is the peripheral structure induced on
  $G$.

  Let $\phi\co G\to G(N_1,\ldots,N_m)$ be a sufficiently long $H$--filling, that
  the conclusions of Theorem \ref{oldstuff} and Theorem
  \ref{heightgoesdown} both hold, and suppose $N_i\dot{\unlhd} P_i$
  for each $i$.

  By Theorem \ref{oldstuff},   $\bar{G}=G(N_1,\ldots,N_m)$ is hyperbolic
  relative to $\bar{P}=\{P_1/N_1,\ldots,P_m/N_m\}$, and
  $(\bar{H},\bar{\mc{D}})$ is relatively quasiconvex in $(\bar{G},\bar{\mc{P}})$.
  Since all the
  peripheral subgroups are finite, $\bar{G}$ is hyperbolic, and
  $\bar{H}$ is a quasiconvex subgroup of $\bar{G}$.  By Theorem
  \ref{heightgoesdown}, the height of $\bar{H}$ in $\bar{G}$ is at
  most $k-1$.
\end{proof}

The next theorem is the same as \cite[Theorem 13.1]{Wise11}, but we
point out some simplifications to the proof.
\begin{theorem}
Let $G=A\ast_C B$ be an amalgamated free product with $C$ quasiconvex in $G$, and $G$ hyperbolic.
Suppose that $A, B$ are virtually special.  Then $G$ is virtually special. There is a similar
statement in the case that $G=A\ast_B$ where $B$ is quasiconvex in $G$, and $A,B$ are virtually special. 
\end{theorem}
\begin{proof}
We will focus on the amalgamated
product case; the HNN case is similar, or follows as a corollary by a doubling trick.
As in the proof of \cite[Theorem 13.1]{Wise11}, it suffices to prove that $C<G$ is separable.   Let  $g\notin C$ be an element that we would like to separate from $C$.  By Theorem 
\ref{oldstuff}.\eqref{separateg}, for all sufficiently long $C$-fillings  $\phi:G\to \overline{G}$, $\phi(g)\notin \phi(C)$.  

Now, take a sequence of fillings as in the proof of Theorem \ref{filling}, ensuring at each stage that the filling maps separate $g$ from the image of $C$, so that we obtain a quotient $\phi \co G \onto \bar{G}$ so that:
\begin{enumerate}
\item $\bar{G}$ is hyperbolic;
\item $\bar{C} = \phi(C)$ is finite; and
\item $\phi(g)\notin \phi(C)$;
\end{enumerate}
The maps $\phi_A = \phi|_A \co A \to \bar{A}$ and $\phi_B \co B \to \bar{B}$ are fillings of $A$ and $B$ respectively.  
By choosing the sequence of fillings defining $\bar{G}$ to be sufficiently long and in appropriately chosen subgroups as in the proof of Theorem \ref{filling}, 
we can ensure that at each stage the maps restricted to the images of $A$ and $B$ satisfy the hypotheses of Theorem \ref{malnormal}.  
Therefore, for a long enough filling we have $\bar{A}$ and $\bar{B}$ are virtually special.

Each of the fillings kernels (taken successively) are normally generated by elements that are in the image of $C$.   Therefore, by taking the obvious presentation for $G$ as an amalgamated free product, and adding relations from $C$ to get $\bar{G}$, we see that
\[	\bar{G} = \bar{A} \ast_{\bar{C}} \bar{B}	.	\]
Since $\bar{C}$ is finite, it is almost malnormal and quasiconvex, so
by the Malnormal Special Combination Theorem
\cite[Theorem 11.3]{Wise11} $\bar{G}$ is virtually special, and hence residually finite.  Let $\eta \co \bar{G} \to Q$ be a homomorphism to a finite group $Q$ so that $\eta(\phi(g))\notin \eta(\bar{C})$.  It is clear that the kernel of $\eta \circ \phi$ is a torsion-free subgroup of $G$ of finite-index separating $g$ from $C$.

Now we finish the proof as in \cite[Theorem 13.1]{Wise11}. Since $C$ is quasiconvex and separable in $G$, there is a finite-index normal subgroup $G'\dot{\lhd} G$ in which $C'=C\cap G'$ is 
malnormal in $G'$. Then $G'$ has a graph-of-groups decomposition with
virtually special vertex groups and malnormal edge groups. 
In particular, $G'$ has a malnormal hierarchy terminating in virtually
special groups, so $G'$ is virtually special by \cite[Theorem
  11.2]{Wise11}.
\end{proof}

\begin{theorem} \label{QVH} \cite[Theorem 13.3]{Wise11}
Every word-hyperbolic group in $\mc{QVH}$ is virtually special. 
\end{theorem}

\bibliographystyle{hamsplain.bst}
\bibliography{refs}

\def\cprime{$'$} \def\cprime{$'$} \def\cprime{$'$}
\providecommand{\bysame}{\leavevmode\hbox to3em{\hrulefill}\thinspace}
\providecommand{\href}[2]{#2}
\begin{thebibliography}{10}

\bibitem{AGM08}
Ian Agol, Daniel Groves, and Jason~Fox Manning, \emph{Residual finiteness,
  {QCERF} and fillings of hyperbolic groups}, Geometry and Topology \textbf{13}
  (2009), 1043--1073, \mbox{arXiv:0802.0709}.

\bibitem{Baker89}
Mark~D. Baker, \emph{Covers of {D}ehn fillings on once-punctured torus
  bundles}, Proc. Amer. Math. Soc. \textbf{105} (1989), no.~3, 747--754.

\bibitem{Baker90}
\bysame, \emph{Covers of {D}ehn fillings on once-punctured torus bundles.
  {II}}, Proc. Amer. Math. Soc. \textbf{110} (1990), no.~4, 1099--1108.

\bibitem{Baker91}
\bysame, \emph{On coverings of figure eight knot surgeries}, Pacific J. Math.
  \textbf{150} (1991), no.~2, 215--228.

\bibitem{BergeronWise09}
Nicolas Bergeron and Dani Wise, \emph{A boundary criterion for cubulation},
  preprint, \mbox{0908.3609}.

\bibitem{BestvinaFeighn92}
M.~Bestvina and M.~Feighn, \emph{A combination theorem for negatively curved
  groups}, J. Differential Geom. \textbf{35} (1992), no.~1, 85--101.

\bibitem{BZ00}
S.~Boyer and X.~Zhang, \emph{Virtual {H}aken {$3$}-manifolds and {D}ehn
  filling}, Topology \textbf{39} (2000), no.~1, 103--114.

\bibitem{BridsonHaefliger}
Martin~R. Bridson and Andr{\'e} Haefliger, \emph{Metric spaces of non-positive
  curvature}, Grundlehren der Mathematischen Wissenschaften [Fundamental
  Principles of Mathematical Sciences], vol. 319, Springer-Verlag, Berlin,
  1999.

\bibitem{BurgerMozes}
Marc Burger and Shahar Mozes, \emph{Lattices in product of trees}, Inst. Hautes
  \'Etudes Sci. Publ. Math. (2000), no.~92, 151--194 (2001).

\bibitem{CL}
D.~Cooper and D.~D. Long, \emph{Virtually {H}aken {D}ehn-filling}, J.
  Differential Geom. \textbf{52} (1999), no.~1, 173--187.

\bibitem{CooperWalsh06I}
Daryl Cooper and Genevieve~S. Walsh, \emph{Three-manifolds, virtual homology,
  and group determinants}, Geom. Topol. \textbf{10} (2006), 2247--2269
  (electronic).

\bibitem{CooperWalsh06II}
\bysame, \emph{Virtually {H}aken fillings and semi-bundles}, Geom. Topol.
  \textbf{10} (2006), 2237--2245 (electronic).

\bibitem{DGO}
Francois Dahmani, Vincent Guirardel, and Denis Osin, \emph{Hyperbolically
  embedded subgroups and rotating families in groups acting on hyperbolic
  spaces}, preprint, November 2011, \mbox{arXiv:1111.7048}.

\bibitem{DT03}
Nathan~M. Dunfield and William~P. Thurston, \emph{The virtual {H}aken
  conjecture: experiments and examples}, Geom. Topol. \textbf{7} (2003),
  399--441 (electronic).

\bibitem{gmrs}
Rita Gitik, Mahan Mitra, Eliyahu Rips, and Michah Sageev, \emph{Widths of
  subgroups}, Trans. Amer. Math. Soc. \textbf{350} (1998), no.~1, 321--329.

\bibitem{gromov:wordhyperbolic}
Mikhael Gromov, \emph{Hyperbolic groups}, Essays in Group Theory (S.~M.
  Gersten, ed.), Mathematical Sciences Research Institute Publications, vol.~8,
  Springer--Verlag, New York, 1987, pp.~75--264.

\bibitem{rhds}
Daniel Groves and Jason~Fox Manning, \emph{{Dehn filling in relatively
  hyperbolic groups}}, Israel Journal of Mathematics \textbf{168} (2008),
  317--429.

\bibitem{Haefliger90}
Andr{\'e} Haefliger, \emph{Orbi-espaces}, Sur les groupes hyperboliques
  d'apr\`es {M}ikhael {G}romov ({B}ern, 1988), Progr. Math., vol.~83,
  Birkh\"auser Boston, Boston, MA, 1990, pp.~203--213.

\bibitem{Haefliger91}
\bysame, \emph{Complexes of groups and orbihedra}, Group theory from a
  geometrical viewpoint ({T}rieste, 1990), World Sci. Publ., River Edge, NJ,
  1991, pp.~504--540.

\bibitem{HaglundWise09}
Frederic Haglund and Dani Wise, \emph{A combination theorem for special cube
  complexes}, preprint, 2009.

\bibitem{HaglundWise07}
Frederic Haglund and Daniel Wise, \emph{Special cube complexes}, Geom. Funct.
  Anal. (2007), 1--69.

\bibitem{Haken62}
Wolfgang Haken, \emph{\"{U}ber das {H}om\"oomorphieproblem der
  3-{M}annigfaltigkeiten. {I}}, Math. Z. \textbf{80} (1962), 89--120.

\bibitem{Hruska}
G.~Christopher Hruska, \emph{Relative hyperbolicity and relative quasiconvexity
  for countable groups}, Algebr. Geom. Topol. \textbf{10} (2010), no.~3,
  1807--1856.

\bibitem{KM09}
Jeremy Kahn and Vladimir Markovic, \emph{Immersing almost geodesic surfaces in
  a closed hyperbolic three manifold}, preprint, 2009, \mbox{arXiv:0910.5501}.

\bibitem{Kapovich01}
Ilya Kapovich, \emph{The combination theorem and quasiconvexity}, Internat. J.
  Algebra Comput. \textbf{11} (2001), no.~2, 185--216.

\bibitem{Kirby}
Robion Kirby, \emph{Problems in low-dimensional topology}, Geometric topology
  (Athens, GA, 1993) (Rob Kirby, ed.), AMS/IP Stud. Adv. Math., vol.~2, Amer.
  Math. Soc., Providence, RI, 1997, pp.~35--473.

\bibitem{KojimaLong88}
S.~Kojima and D.~D. Long, \emph{Virtual {B}etti numbers of some hyperbolic
  {$3$}-manifolds}, A f\^ete of topology, Academic Press, Boston, MA, 1988,
  pp.~417--437.

\bibitem{Lackenby06}
Marc Lackenby, \emph{Heegaard splittings, the virtually {H}aken conjecture and
  property {$(\tau)$}}, Invent. Math. \textbf{164} (2006), no.~2, 317--359.

\bibitem{Lackenby07}
\bysame, \emph{Some 3-manifolds and 3-orbifolds with large fundamental group},
  Proc. Amer. Math. Soc. \textbf{135} (2007), no.~10, 3393--3402 (electronic).

\bibitem{Long87}
D.~D. Long, \emph{Immersions and embeddings of totally geodesic surfaces},
  Bull. London Math. Soc. \textbf{19} (1987), no.~5, 481--484.

\bibitem{MM-P}
Jason~Fox Manning and Eduardo Mart{\'{\i}}nez-Pedroza, \emph{Separation of
  relatively quasiconvex subgroups}, Pacific J. Math. \textbf{244} (2010),
  no.~2, 309--334.

\bibitem{mp2}
Eduardo Mart\'{\i}nez-Pedroza, \emph{On quasiconvexity and relatively
  hyperbolic structures on groups}, Geometriae Dedicata, 1--22,
  10.1007/s10711-011-9610-3.

\bibitem{Masters00}
Joseph~D. Masters, \emph{Virtual homology of surgered torus bundles}, Pacific
  J. Math. \textbf{195} (2000), no.~1, 205--223.

\bibitem{Masters07}
\bysame, \emph{Virtually {H}aken surgeries on once-punctured torus bundles},
  Comm. Anal. Geom. \textbf{15} (2007), no.~4, 733--756.

\bibitem{MorganTian08}
John Morgan and Gang Tian, \emph{Completion of the proof of the geometrization
  conjecture}, preprint, 2008, \mbox{arXiv.org:0809.4040}.

\bibitem{Mo}
John~W. Morgan, \emph{On {T}hurston's uniformization theorem for
  three-dimensional manifolds}, The Smith conjecture (New York, 1979), Academic
  Press, Orlando, FL, 1984, pp.~37--125.

\bibitem{Osin07}
Denis~V. Osin, \emph{Peripheral fillings of relatively hyperbolic groups},
  Invent. Math. \textbf{167} (2007), no.~2, 295--326.

\bibitem{Per03}
Grisha Perelman, \emph{{Ricci flow with surgery on three-manifolds}},
  \mbox{\special{html:<a href="http://front.math.ucdavis.edu/math.DG/0303109">}
  arXiv:math.DG/0303109\special{html:</a>}}.

\bibitem{Per02}
\bysame, \emph{{The entropy formula for the Ricci flow and its geometric
  applications}}, \mbox{\special{html:<a
  href="http://front.math.ucdavis.edu/math.DG/0211159">}
  arXiv:math.DG/0211159\special{html:</a>}}.

\bibitem{Sageev95}
Michah Sageev, \emph{Ends of group pairs and non-positively curved cube
  complexes}, Proc. London Math. Soc. (3) \textbf{71} (1995), no.~3, 585--617.

\bibitem{Schwermer04}
Joachim Schwermer, \emph{Special cycles and automorphic forms on arithmetically
  defined hyperbolic 3-manifolds}, Asian J. Math. \textbf{8} (2004), no.~4,
  837--859.

\bibitem{Stallings83}
John~R. Stallings, \emph{Topology of finite graphs}, Invent. Math. \textbf{71}
  (1983), no.~3, 551--565.

\bibitem{Th:82}
William~P. Thurston, \emph{Three-dimensional manifolds, {K}leinian groups and
  hyperbolic geometry}, Bull. Amer. Math. Soc. (N.S.) \textbf{6} (1982), no.~3,
  357--381.

\bibitem{Waldhausen}
Friedhelm Waldhausen, \emph{On irreducible {$3$}-manifolds which are
  sufficiently large}, Ann. of Math. (2) \textbf{87} (1968), 56--88.

\bibitem{Wise11}
Daniel Wise, \emph{The structure of groups with a quasiconvex hierarchy},
  preprint, 2011.

\end{thebibliography}

\end{document}